\documentclass[lettersize,journal]{IEEEtran}
\usepackage{colortbl}
\usepackage[table,xcdraw]{xcolor} 
\definecolor{customblue}{rgb}{0.85, 0.89, 0.98} 
\usepackage{amsmath,amsfonts}
\usepackage{amssymb}
\usepackage{mathtools}
\usepackage{amsthm}
\usepackage{subcaption}
\usepackage{tikz}
\usetikzlibrary{calc}
\usepackage{algorithmic}
\usepackage{algorithm}

\usepackage{array}
\usepackage[caption=false,font=normalsize,labelfont=sf,textfont=sf]{subfig}
\usepackage{stfloats} 

\usepackage{placeins}
\usepackage{cuted} 
\usepackage{textcomp}
\usepackage{stfloats}
\usepackage{url}
\usepackage{verbatim}
\usepackage{graphicx}

\usepackage{cite}
\DeclareMathOperator{\tr}{tr}
\DeclareMathOperator{\inv}{inv}

\DeclareMathOperator{\diag}{diag}
\newtheorem{definition}{Definition}
\newtheorem{theorem}{Theorem}
\theoremstyle{remark}
\newtheorem{remark}{\textbf{Remark}}

\usepackage{tabularx}
\usepackage{booktabs}
\usepackage{multirow}

\usepackage[font=scriptsize]{caption} 
\usepackage[table]{xcolor}

\usepackage{colortbl}      

\begin{document}

\title{SPP-SBL: Space-Power Prior Sparse Bayesian Learning for Block Sparse Recovery}

\author{Yanhao Zhang,~\IEEEmembership{Graduate Student Member,~IEEE,}~~Zhihan Zhu,~\IEEEmembership{Graduate Student Member,~IEEE,} \\
	and Yong Xia,~\IEEEmembership{Member,~IEEE}
\thanks{This work was supported by the National Key R\&D Program of China (Grant No. 2021YFA1003303), National Natural Science Foundation of China (Grant No. 125B2016), and the Academic Excellence Foundation of BUAA for PhD Students. \textit{(Corresponding authors: Zhihan Zhu; Yong Xia.)}
	
	Yanhao Zhang, Zhihan Zhu and Yong Xia are with the LMIB of the Ministry of Education, School of Mathematical Sciences, Beihang University, Beijing 100191, China (e-mail: \{yanhaozhang, zhihanzhu, yxia\}@buaa.edu.cn).}

}



\maketitle

\begin{abstract}
The recovery of block-sparse signals with unknown structural patterns remains a fundamental challenge in structured sparse signal reconstruction. By proposing a variance transformation framework, this paper unifies existing pattern-based block sparse Bayesian learning methods, and introduces a novel space power prior based on undirected graph models to adaptively capture the unknown patterns of block-sparse signals. By combining the {Expectation-Maximization} algorithm with high-order equation root-solving, we develop a new structured sparse Bayesian learning method, SPP-SBL, which effectively addresses the open problem of space coupling parameter estimation in pattern-based methods. We further demonstrate that learning the relative values of space coupling parameters is key to capturing unknown block-sparse patterns and improving recovery accuracy. Experiments validate that SPP-SBL successfully recovers various challenging structured sparse signals (e.g., chain-structured signals and multi-pattern sparse signals) and real-world multi-modal structured sparse signals (images, audio), showing significant advantages in recovery accuracy across multiple metrics.
\end{abstract}


\section{Introduction}\label{intro}
\IEEEPARstart{S}{parse}
 recovery through Compressed Sensing (CS) has garnered significant attention due to its robust theoretical foundation and wide-ranging applications \cite{donoho2006compressed}, particularly for its efficacy in reconstructing sparse vectors from a substantially reduced number of linear measurements. The general model with noise is represented as follows:
\begin{equation}
	\mathbf{y} = \mathbf{\Phi} \mathbf{x}+ \mathbf{n}, \label{basic CS}
\end{equation}
where  \(\mathbf{y} \in \mathbb{R}^{M \times 1}\) denotes the available measurement vector, \(\mathbf{\Phi} \in \mathbb{R}^{M \times N}\) is a known design matrix \cite{gorodnitsky1997sparse}, and \(\mathbf{x} \in \mathbb{R}^{N \times 1}\) is the sparse signal to be recovered, with \(N \gg M\). Additionally, \(\mathbf{n} \sim \mathcal{N}(0, \gamma^{-1} \mathbf{I})\) represents the additive Gaussian noise. In real-world applications, \(\mathbf{x}\) often exhibits transform sparsity \cite{donoho2006compressed}, meaning it becomes sparse in a suitable transform domain, such as Wavelet, Fourier transforms, etc. This leads to the representation \(\mathbf{x} = \boldsymbol{\Psi} \mathbf{w}\), where \(\mathbf{w}\) is sparse, and the sensing matrix \(\Theta = \boldsymbol{\Phi\Psi}\)  satisfies the Restricted Isometry Property (RIP) \cite{candes2005decoding}. Then we can simply replace $\mathbf{x}$ by $\boldsymbol{\Psi }\mathbf{w}$ in \eqref{basic CS} and solve it in the same way.  Therefore, for convenience, we will assume that \(\mathbf{x}\) inherently possesses a (block) sparse structure in the following discussion.

The above problem has been extensively studied, and numerous algorithms have been developed to provide effective sparse solutions. Classic methods include Lasso  \cite{tibshirani1996regression}, Basis Pursuit \cite{chen2001atomic}, Orthogonal Matching Pursuit (OMP) \cite{pati1993orthogonal}, Sparse Bayesian Learning (SBL) \cite{tipping2001sparse, ji2008bayesian}, etc. However, it has become increasingly clear that relying solely on the sparsity of \(\mathbf{x}\) is often insufficient, especially in scenarios with limited sample sizes \cite{eldar2010block,donoho2013accurate}. In practice, sparse signals often exhibit block-sparse structures, which can be leveraged to enhance sparse recovery performance. Block sparsity is generally defined by the occurrence of non-zero entries grouped into distinct blocks, with only a small subset of these blocks containing non-zero values. Formally, the structure of a block sparse signal \(\mathbf{x}\) with \(g\) blocks can be expressed as
\begin{equation}
	\mathbf{x} = [\mathbf{x}_1^T, \mathbf{x}_2^T, \ldots, \mathbf{x}_g^T]^T,
\end{equation}
where each block \(\mathbf{x}_i\in \mathbb{R}^{L_i}\) may vary in size, and it is assumed that only \(k\) out of the \(g\) blocks are non-zero (\(k \ll g\)).
 	Real-world signals, such as images, audio, and wireless channels, often exhibit block-sparse structures after appropriate transformations, where significant coefficients tend to occur in clustered groups rather than in isolation (see illustrative examples in Appendix~\ref{App:block}). 
 	 For example, natural images often exhibit block sparsity in the discrete wavelet transform (DWT) domain \cite{mishali2009blind, yu2012bayesian, he2009exploiting}, while audio signals typically show block-sparse structures in the discrete cosine transform (DCT) domain \cite{gribonval2003harmonic, NEURIPS2024_ead542f1}. Similarly, in wireless channel estimation, block-sparse patterns naturally emerge in the discrete Fourier transform (DFT) domain \cite{dai2018non}.
 	Block structures also reflect the inherent characteristics of many real-world applications, including gene expression analysis \cite{tibshirani2005sparsity}, radar imaging \cite{lv2014inverse}, and channel estimation \cite{srivastava2018quasi, wang2025near}, where physically or functionally related variables are naturally activated together.

Several approaches have been developed for the reconstruction of block-sparse signals, with early methods typically assuming known block partitions—commonly referred to as \textit{block-based} methods. Representative models such as Group Lasso \cite{yuan2006model}, Group Basis Pursuit \cite{van2011sparse}, Block-OMP \cite{eldar2010block} and Model-based CS \cite{baraniuk2010model}. 
These methods rely on predefined block structures with fixed dimensions, which are often difficult to specify accurately in practical applications.

Block-based formulations have also been extensively studied within Bayesian frameworks by extending classical element-wise sparse priors to the block level \cite{park2008bayesian,carvalho2010horseshoe} . In this line of work, Bayesian group sparsity models, such as the Bayesian Group Lasso \cite{raman2009bayesian} and the Group Horseshoe \cite{hernandez2013generalized}, impose block-level shrinkage and selection through hierarchical priors. Within the signal processing community, block-based Bayesian modeling has been further developed under the Sparse Bayesian Learning (SBL) framework. Representative block-based SBL variants, including Temporally Sparse Bayesian Learning (TSBL) \cite{zhang2011sparse} and Block Sparse Bayesian Learning (BSBL) \cite{zhang2013extension}, explicitly exploit intra-block correlations but still depend on a priori knowledge of block partitions. To alleviate this limitation, the recently proposed Diversified SBL (DivSBL) algorithm \cite{NEURIPS2024_ead542f1} reduces sensitivity to predefined block structures while remaining within the block-based modeling paradigm.

Another class of methods is \textit{pattern-based}, which aims to capture block-sparse structures without predefined block partitions by promoting clustered supports through local dependencies among neighboring elements. Early representatives include structure-aware greedy algorithms such as StructOMP \cite{huang2009learning}, which guides support selection via pattern-driven structural constraints.
Building upon the SBL framework, pattern-based Bayesian models replace explicit block partitions with local coupling mechanisms among neighboring coefficients. Pattern-Coupled Sparse Bayesian Learning (PC-SBL) \cite{fang2014pattern} can be viewed as the first pattern-based formulation within SBL,  introducing a novel framework that couples hyperparameters in the variance layer via a coupling parameter $\beta$, offering an intuitive and effective mechanism for capturing underlying block structures. The Alternative to Extended BSBL (A-EBSBL) method \cite{wang2018alternative} further established a link between pattern-based models and predefined block-based frameworks such as (E)BSBL by employing uniform weighting of neighboring variances (i.e., setting $\beta = 1$). However, the estimation of the coupling parameter $\beta$ remains an open and challenging problem \cite{dai2018non}, as its adaptive determination is generally difficult \cite{wang2018alternative, wang2020structured} and often requires task-specific tuning to fit the underlying block-sparse patterns \cite{sant2022block}.

To circumvent the direct estimation of the coupling parameter ${\beta}$, the Non-uniform Burst Sparsity algorithm \cite{dai2018non} discretizes ${\beta}$ into binary values (0 or 1) for estimation. Similarly, the Robust Cluster Structured SBL (RCS-SBL) method \cite{wang2020structured} reformulates ${\beta}$ as an indicator function $I_i$, which denotes whether coupling with neighboring elements is required and also takes binary values. Recently, the TV-SBL algorithm \cite{sant2022block} introduces total variation (TV) regularization at the hyperparameter level to avoid the need for direct estimation of ${\beta}$.


Therefore, estimating the coupling parameter in pattern-based algorithms remains an open problem. The inability to effectively estimate \(\vec{\boldsymbol{\beta}}\) not only leads to boundary overestimation \cite{dai2018non,wang2018alternative, wang2020structured} but also constrains the adaptability of these methods to more complex data structures, such as multi-pattern datasets \cite{sant2022block} and chain-structured datasets \cite{korki2016iterative}.

In this paper, we propose a unified variance transformation framework for pattern-based algorithms and introduce a space-power prior based on an undirected graph to perform Bayesian estimation of the coupling parameters in the transformation matrix. Notably, within this framework, each coupling parameter $\beta_i$
admits a unique positive real solution. Leveraging this property, we develop an efficient block sparse Bayesian learning algorithm, SPP-SBL, based on the expectation–maximization (EM) algorithm. This approach provides an effective estimation scheme for $\vec{\boldsymbol{\beta}}$, thereby addressing a long-standing open problem in the field.

Importantly, we observe that while the absolute values of the elements in $\vec{\boldsymbol{\beta}}$ have limited impact on sparse recovery performance, their relative magnitudes are crucial. Our results demonstrate that adaptively learning $\vec{\boldsymbol{\beta}}$ to reflect underlying structural patterns significantly outperforms existing methods that assume a fixed and shared $\beta$. Furthermore, by resolving the open problem of $\vec{\boldsymbol{\beta}}$ estimation, the proposed framework simultaneously addresses several core challenges in structured sparsity modeling, including boundary overestimation and limitations in handling multi-pattern or chain-structured data. Both theoretical analysis and experimental results substantiate the superiority of the proposed SPP-SBL algorithm.

The rest of the paper is organized as follows.
Section \ref{sec: variance} revisits existing pattern-based models from the perspective of a unified variance transformation framework. Within this framework, a specific transformation matrix—the Symmetric Diversified Coupling Matrix $\boldsymbol{T}_{\rm SPP}$, based on an undirected graph, is introduced.
Section \ref{SPP model} presents the space-power prior, which addresses the open problem of estimating coupling parameters in pattern-based method. The corresponding Bayesian posterior inference algorithm, along with a uniqueness theorem and an upper-bound theory for the coupling parameter estimation, is provided in Section \ref{inference}.
Section \ref{sec: discussion} discusses the benefits of accurate coupling parameter estimation, particularly in reducing boundary overestimation and overcoming limitations in estimating complex pattern signals.
Section \ref{experiment} reports experimental results. Conclusions and future works are summarized in Section \ref{conclusion}.

Notation: Throughout the paper, bold lowercase and uppercase letters denote vectors and matrices, respectively. The $i$-th element of a vector $\mathbf{x}$ is denoted by $x_i$. The notation  $\operatorname{diag}(\boldsymbol{\alpha})$ represents a diagonal matrix whose diagonal elements are given by the vector $\boldsymbol{\alpha}$. The symbol $\mathcal{N}(\cdot)$ denotes the multivariate Gaussian distribution.
To avoid ambiguity, an overhead arrow is used exclusively for the vector $\vec{\boldsymbol{\beta}}$ to distinguish it from the scalar $\beta$.

\section{Variance Transformation Framework} \label{sec: variance}

\subsection{Revisit Pattern-coupled Priors}

Before introducing the variance transformation framework, we first review four classical priors employed in pattern-based Bayesian models. Generally, these priors can be uniformly expressed as coupling priors, where the dependencies  are determined by adjacent elements within the hyperparameter \( \boldsymbol{\alpha} \):
\begin{equation}
	p(\mathbf{x} \mid \boldsymbol{\alpha}) = \prod_{i=1}^{N} p\left(x_i \mid \alpha_i, \alpha_{i+1}, \alpha_{i-1}\right).
\end{equation}
Different coupling schemes yield distinct models and algorithms. 
	For example, by connecting adjacent hyperparameters using one coupling parameter \(\beta\), 
PC-SBL:
	\begin{equation}
		p\left(x_i \mid \alpha_i, \alpha_{i+1}, \alpha_{i-1}\right)=\mathcal{N}\left( 0, \left(\alpha_i + \beta\,\alpha_{i+1} + \beta\,\alpha_{i-1}\right)^{-1}\right),\label{PC-SBL}
	\end{equation}
and  A-EBSBL:
	\begin{equation}
		p\left(x_i \mid \alpha_i, \alpha_{i+1}, \alpha_{i-1}\right)=\mathcal{N}\left( 0,\,\alpha_i^{-1} + \beta\,\alpha_{i+1}^{-1} + \beta\,\alpha_{i-1}^{-1}\right), \label{AEBSBL}
	\end{equation}
	are derived.
These two models intuitively couple at the variance layer. 
Additionally, the Non-Uniform Burst-Sparsity prior is designed as follows:
\begin{equation}
	\begin{aligned}
		&p\left(x_i \mid \alpha_i, \alpha_{i+1}, \alpha_{i-1}, \mathbf{z}_i\right)\\
		=&\left\{\mathcal{N}\left( 0,\,\alpha_{i-1}^{-1}\right)\right\}^{z_{i,1}}
		\cdot 
		\left\{\mathcal{N}\left( 0,\,\alpha_i^{-1}\right)\right\}^{z_{i,2}}
	\cdot\left\{\mathcal{N}\left( 0,\,\alpha_{i+1}^{-1}\right)\right\}^{z_{i,3}},\label{non-uniform}
	\end{aligned}
\end{equation}
where \(z_{i,j}\in\{0,1\}\) and \(\sum_{j=1}^{3} z_{i,j}=1\). We can perform a straightforward transformation of \eqref{non-uniform}, which reveals that it remains a variance-coupling model:
	\begin{align}
		\eqref{non-uniform}
		&\propto \exp\left(-\frac{1}{2}\,z_{i,1}\,\alpha_{i-1}\,x_i^2\right)
		\cdot \exp\left(-\frac{1}{2}\,z_{i,2}\,\alpha_i\,x_i^2\right)\notag\\
		&\quad\cdot \exp\left(-\frac{1}{2}\,z_{i,3}\,\alpha_{i+1}\,x_i^2\right)\notag\\
		&=\exp\left[-\frac{1}{2} \left(z_{i,1}\,\alpha_{i-1}+z_{i,2}\,\alpha_i+z_{i,3}\,\alpha_{i+1}\right)x_i^2\right]\notag\\
		&\propto \mathcal{N}\left( 0,\,\left(z_{i,2}\,\alpha_i+z_{i,3}\,\alpha_{i+1}+z_{i,1}\,\alpha_{i-1}\right)^{-1}\right),\label{Non-uniform}
	\end{align}
 where the coupling parameter is binary. Likewise, the  prior in RCS-SBL is as follows:
 	\begin{equation}
 	\begin{aligned}
 		&p(x_i \mid I_i)\\
 		=& \mathcal{N}\Big(0,\big(\alpha_{i-1}+\alpha_i+\alpha_{i+1}\big)^{-1}\Big)^{1\left[I_i=1\right]}
 		\mathcal{N}\Big(0,\,\alpha_i^{-1}\Big)^{1\left[I_i=0\right]}.\label{RCS-SBL}
\end{aligned}
\end{equation}
Similarly, we can reformulate it as
	\begin{align}
		\eqref{RCS-SBL}
		&\propto \exp\left(-\frac{1\left[I_i=1\right]}{2} \left(\alpha_{i-1}+\alpha_i+\alpha_{i+1}\right)x_i^2\right)\notag\\
		&\qquad \cdot \exp\left(-\frac{1\left[I_i=0\right]}{2} \alpha_i x_i^2\right)\notag\\
		&=\mathcal{N}\left( 0,\,\left(\alpha_i+1\left[I_i=1\right] (\alpha_{i-1}+\alpha_{i+1})\right)^{-1}\right).\label{RCS-SBL2}
	\end{align}
Here, the coupling parameter is an indicator function, where \( I_i = 1 \) signifies that information from both adjacent elements is incorporated, i.e., \( 1[\cdot] = 1 \); otherwise, it is 0. Consequently, the coupling parameter remains a binary variable, enforcing equal dependence on both neighboring elements.

Based on \eqref{PC-SBL}, \eqref{AEBSBL}, \eqref{Non-uniform}, and \eqref{RCS-SBL2}, we propose a \emph{variance transformation framework} for pattern-coupled Bayesian models.

\begin{figure*}[t]
	\centering
	\begin{footnotesize}
		\begin{equation}
			\hspace{-0.4mm}\boldsymbol{T}_1 =
			\begin{bmatrix}
				1     & \beta & 0     & \cdots     & 0     \\
				\beta & 1     & \beta & \cdots     & 0     \\
				0     & \beta & 1     & \cdots     & 0     \\
				\vdots& \vdots& \vdots& \ddots & \vdots\\
				0     & 0     & 0     & \cdots  &\beta \\
				0     & 0     & 0     & \cdots  & 1     \\
			\end{bmatrix}, 
			\boldsymbol{T}_2 =
			\begin{bmatrix}
				z_{12} & z_{13} & 0     & \cdots    & 0     \\
				z_{21} & z_{22} & z_{23} & \cdots     & 0     \\
				0      & z_{31} & z_{32} & \cdots     & 0     \\
				\vdots & \vdots & \vdots & \ddots &  \vdots\\
				0      & 0      & 0      & \cdots &  z_{N-1,3}\\
				0      & 0      & 0      & \cdots &  z_{N,2}   \\
			\end{bmatrix}, 
			\boldsymbol{T}_3 = \begin{bmatrix}
				1 & 1\left[I_1=1\right] & 0     & \cdots      & 0     \\
				1\left[I_2=1\right] & 1 & 1\left[I_2=1\right] & \cdots     & 0     \\
				0 & 1\left[I_3=1\right] & 1 & \cdots      & 0     \\
				\vdots & \vdots & \vdots & \ddots  & \vdots\\
				0 & 0 & 0 & \cdots  & 1\left[I_{N-1}=1\right] \\
				0 & 0 & 0 & \cdots  & 1     \\
			\end{bmatrix}. \label{T}
		\end{equation}
	\end{footnotesize}
	\vspace{-5mm}
\end{figure*}

\begin{figure*}[t]
	\centering
	\includegraphics[width=7.2in]{figs/undirected_graph2}
	\caption{Graph structures of pattern coupling. \textbf{(A)} The underlying graph structures corresponding to classical pattern-coupled models. \textbf{(B)} The proposed Space-Power-Prior (SPP) coupling model: \textbf{(B1)} uniform interactions between adjacent nodes, \textbf{(B2)} edge-specific parameters enabling boundary-aware learning, \textbf{(B3)} the resulting symmetric diversified coupling matrix \( \boldsymbol{T}_{\text{SPP}} \), and \textbf{(B4)} the overall Bayesian hierarchical model.}
	\label{directed graph}
\end{figure*}

\subsection{A Unified Variance Transformation Framework}

We begin by introducing an  inverse operator for vectors.
\begin{definition}
	For any vector \( \mathbf{w}\in\mathbb{R}^N \), we define the operator \(\inv(\mathbf{w})\) as
	\begin{equation}
		\inv(\mathbf{w}) = \left(w_1^{-1},\,w_2^{-1},\,\cdots,\,w_N^{-1}\right).
	\end{equation}
\end{definition}

From \eqref{PC-SBL}, \eqref{AEBSBL}, \eqref{Non-uniform}, and \eqref{RCS-SBL2}, we observe that the coupling prior is designed by applying a linear transformation to the second-order information (i.e., the variance or its inverse). Motivated by this observation, we introduce a unified framework for pattern-based Bayesian methods.

\textbf{The Variance Transformation Framework:}

Let \( \boldsymbol{T} \in \mathbb{R}^{N \times N} \) be the \emph{coupling matrix} (or \emph{transformation matrix}), which encodes the structured dependencies among variances. Then, the pattern-coupled sparse Bayesian priors can be formulated as:  
\begin{equation}
	p(\mathbf{x}\mid  \boldsymbol{\alpha}, \Theta(\boldsymbol{T}))=\mathcal{N}\left(\mathbf{0}, 	\boldsymbol{\Sigma}_{\text{prior}} \right),\label{vt1}
\end{equation}
where the covariance matrix is given by
\begin{equation}
	\boldsymbol{\Sigma}_{\text{prior}} = \diag\left(\inv\left(\boldsymbol{T}\boldsymbol{\alpha}\right)\right)
	\quad \text{or} \quad
	\diag\left(\boldsymbol{T}\inv\left(\boldsymbol{\alpha}\right)\right),\label{vt2}
\end{equation}
and $\Theta(\boldsymbol{T})$ represents the hyperparameters in $\boldsymbol{T}$. This formulation provides a systematic approach to modeling structured dependencies in sparse Bayesian learning.
The following examples illustrate specific coupling schemes within this framework, where $\boldsymbol{T}_1$ to $\boldsymbol{T}_3$
are defined in \eqref{T}.

\noindent Example 1 (PC-SBL):
\begin{equation}
	p(\mathbf{x} \mid \boldsymbol{\alpha}, \beta) = \mathcal{N}\left(\mathbf{0},\,\diag\left(\inv\left(\boldsymbol{T}_1\boldsymbol{\alpha}\right)\right)\right).
\end{equation}

\noindent  Example 2 (A-EBSBL):
\begin{equation}
	p(\mathbf{x} \mid \boldsymbol{\alpha}, \beta) = \mathcal{N}\left(\mathbf{0},\,\diag\left(\boldsymbol{T}_1\inv\left(\boldsymbol{\alpha}\right)\right)\right).
\end{equation}

\noindent Example 3 (Non-Uniform Burst Sparsity):
\begin{equation}
	p(\mathbf{x} \mid \boldsymbol{\alpha}, \Theta(\boldsymbol{T}_2)) = \mathcal{N}\left(\mathbf{0},\,\diag\left(\inv\left(\boldsymbol{T}_2\boldsymbol{\alpha}\right)\right)\right).
\end{equation}

\noindent Example 4 (RCS-SBL):
\begin{equation}
	p(\mathbf{x} \mid \boldsymbol{\alpha}, \Theta(\boldsymbol{T}_3)) = \mathcal{N}\left(\mathbf{0},\,\diag\left(\inv\left(\boldsymbol{T}_3\boldsymbol{\alpha}\right)\right)\right).
\end{equation}

These examples not only demonstrate the versatility of the unified variance transformation framework but also illustrate how different coupling matrices can capture diverse structural dependencies in the data.
Notably, the effectiveness of a pattern-based model is highly contingent upon the design of \( \boldsymbol{T} \). This raises a key question: \emph{Which structural design of \( \boldsymbol{T} \) is more appropriate for capturing the intrinsic correlations of block sparse signals?}

\subsection{The Symmetric Diversified Coupling Matrix $\boldsymbol{T}_{\rm SPP}$}


This section analyzes the spatial coupling relationships of block sparse signals from the perspective of graph structures. 

Fig. \ref{directed graph} (\textbf{A}) depicts the inherent graph structures corresponding to classical pattern-based models. As shown in \textbf{A1}, PC-SBL$\backslash$A-EBSBL essentially adopts an undirected graph, where the mutual influence between adjacent nodes is governed by a uniform coupling parameter, $\beta$. However, controlling all edges using the same $\beta$ is overly coarse. Moreover, as mentioned in Section \ref{intro}, even for a single coupling parameter $\beta$, its estimation remains an open problem, and it is typically set to $\beta=1$ by default.  

The models in \textbf{A2} and \textbf{A3} can be viewed as directed graphs. While these models introduce diversity in the coupling parameters between adjacent nodes, they remain limited in two key aspects. First, the coupling parameters are binary, imposing a restrictive structure. More importantly, the directed nature of these graphs results in asymmetric mutual influences between adjacent nodes, as indicated by the difference between the upper red solid lines and the corresponding lower blue dashed lines. This asymmetry is inconsistent with the principles of mutual influence.

Therefore, we propose a coupling mechanism where the mutual influence between adjacent nodes is symmetric, and the coupling parameter \( \beta \) is extended to a diversified vector \( 
\vec{\boldsymbol{\beta}}\in \mathbb{R}^{N-1} \). This formulation enables boundary-aware structure learning by adapting $\vec{\boldsymbol{\beta}}$, as illustrated in Fig.  \ref{directed graph}(\textbf{B}). Incorporating these two aspects, we introduce the symmetric diversified coupling matrix \( \boldsymbol{T}_{\text{SPP}} \) (SPP: Space-Power Prior, which will be discussed in Section \ref{SPP model}) as: 
	\begin{equation}   
		\boldsymbol{T}_{\rm SPP} =
		\begin{bmatrix}
			1     & \beta_1 & 0     & \cdots     &0& 0     \\
			\beta_1 & 1     & \beta_2 & \cdots    &0 & 0     \\
			0     & \beta_2 & 1     & \cdots  &0   & 0     \\
			\vdots& \vdots& \vdots& \ddots &\vdots& \vdots\\
			0     & 0     & 0     & \cdots  &1&\beta_{N-1} \\
			0     & 0     & 0     & \cdots &\beta_{N-1}  & 1     \\
		\end{bmatrix}.
	\end{equation}

Therefore, under the variance transformation framework induced by the coupling matrix \( \boldsymbol{T}_{\text{SPP}} \), the element \( x_i \) at position \( i \) depends on five hyperparameters: \( \alpha_i, \alpha_{i+1}, \alpha_{i-1}, \beta_{i-1}, \beta_i \).

\vspace{2mm}

\begin{remark}
	 Within the variance transformation framework, we argue that a coupling matrix \( \boldsymbol{T} \) satisfying (i) invertibility, (ii) symmetry, and (iii) appropriately diversified parameters provides a structurally coherent representation of spatial dependencies among nodes. The coupling matrix $\boldsymbol{T}_{\rm SPP}$ is designed to capture these dependencies while maintaining model parsimony. Nonetheless, alternative coupling matrix designs could be explored within this framework, though such extensions are out of the scope of this paper.
\end{remark}

\section{Space-Power Prior Bayesian model}\label{SPP model}

In this section, we propose a hierarchical Bayesian model based on \( \boldsymbol{T}_{\text{SPP}} \) coupling matrix for block sparse recovery. 

\subsection{Space-Power Prior}
In this model, each element \( x_i \) depends not only on the variance-layer hyperparameters \( \alpha_i, \alpha_{i+1}, \alpha_{i-1} \) but also on the coupling parameters \( \beta_{i-1}, \beta_i \). Specifically, let \( \vec{\boldsymbol{\beta}} \in \mathbb{R}^{N-1} \) with boundary conditions \( \alpha_0 = \alpha_{N+1} = 0 \), \( \beta_0 = \beta_N = 0 \),  then the prior of \( \mathbf{x} \) is given by:
\begin{equation}
	p(\mathbf{x} \mid \boldsymbol{\alpha},\vec{\boldsymbol{\beta}})=\prod_{i=1}^N p\left(x_i \mid \alpha_i, \alpha_{i+1}, \alpha_{i-1}, \beta_{i-1}, \beta_{i}\right),\label{spp1}
\end{equation}
where 
\begin{align}
	&p\left(x_i \mid \alpha_i, \alpha_{i+1}, \alpha_{i-1}, \beta_{i-1}, \beta_{i}\right) \notag\\
	=&\mathcal{N}\left(  0,\left(\alpha_i+\beta_{i-1} \alpha_{i-1}+\beta_i \alpha_{i+1}\right)^{-1}\right)\label{spp2}\\
	=&\{\mathcal{N}\left(0, \alpha_i^{-1}\right)\} \cdot\{\mathcal{N}\left(0, \alpha_{i-1}^{-1}\right)\}^{\beta_{i-1}} \cdot\{\mathcal{N}\left(0, \alpha_{i+1}^{-1}\right)\}^{\beta_{i}}.\label{spp3}
\end{align}
From \eqref{spp3}, the prior of \( x_i \) depends not only on its own hyperparameter \( \alpha_i \) but also on the neighboring hyperparameters \{\( \alpha_{i-1} \), \( \alpha_{i+1} \)\} in space, with the dependence regulated by the power terms \( \{\beta_{i-1}, \beta_{i}\} \). This is why we refer to the prior of \( x_i \) as the \textit{Space-Power prior} (SPP).  

Reformulating \eqref{spp1} \eqref{spp2} within the variance transformation framework \eqref{vt1} \eqref{vt2}, the prior of \( \mathbf{x} \) can be expressed as:  
\begin{equation}
	p(\mathbf{x} \mid \boldsymbol{\alpha}, \vec{\boldsymbol{\beta}}) = \mathcal{N}(\mathbf{0},  \diag(\inv(\boldsymbol{T}_{\rm SPP}\boldsymbol{\alpha}))).\label{prior_x}
\end{equation}

\begin{remark}
	When $\vec{\boldsymbol{\beta}} = \mathbf{0}$, the model simplifies to standard SBL, and when all \( \beta_i \) are identical, it reduces to PC-SBL.
\end{remark}

\vspace{2mm}

 Notably, \( \beta_i \) influences the prior variance through two terms: \( \beta_i \alpha_{i+1} \) in the variance of \( x_i \) and \( \beta_i \alpha_i \) in that of \( x_{i+1} \). Consequently, learning \( \beta_i \) effectively captures the mutual dependence between \( x_i \) and \( x_{i+1} \): 
As an illustrative example from \cite{fang2014pattern}, when \( \alpha_i \) approaches infinity, the corresponding coefficient \( x_i \) is forced to zero. Moreover, since \( \alpha_i \) also contributes to the priors of its neighbors \{\( x_{i+1} \), \( x_{i-1} \)\},  PC-SBL enforces a similar shrinkage effect on them. However, if \( x_{i+1} \) lies at the boundary of a nonzero block, this rigid coupling with a single $\beta$ may cause misidentify. In contrast, by learning \( \beta_i \), our model explicitly determines the extent of dependency between \( x_i \) and \( x_{i+1} \), allowing it to adaptively mitigate boundary misidentification. {Section \ref{sec: discussion}} will further discuss how this learning mechanism facilitates a boundary-aware structure for \( \mathbf{x} \).

We now present the prior design for the hyperparameters \( \boldsymbol{\alpha} \) and $\vec{\boldsymbol{\beta}}$. For \( \boldsymbol{\alpha} \), we follow a traditional sparse Bayesian learning approach and use a Gamma prior, defined as:
\begin{equation}
	p(\boldsymbol{\alpha})=\prod_{i=1}^{N}\text{Gamma}\left(\alpha_i \mid a,b\right)= \prod_{i=1}^{N}\frac{b^a}{\Gamma(a)}\alpha_i^{a-1} e^{-b\alpha_i},\label{prior_alpha}
\end{equation}
 where \( \Gamma(\cdot) \) is the Gamma function, and \( a \) and \( b \) are typically set to small values, such as $10^{-4}$, to ensure non-informativeness of the priors \cite{tipping2001sparse}. 
As discussed in \cite{tipping2001sparse}, this Gaussian-inverse hierarchical prior promotes a learning process that automatically deactivates most coefficients considered irrelevant, leaving only a few significant ones to explain the data. Therefore, the characterization of the block structure primarily depends on the design of the coupling parameters $\vec{\boldsymbol{\beta}}$.

In order to select an appropriate distribution for coupling parameter $\vec{\boldsymbol{\beta}}$ in Space-Power prior, we first highlight an important observation (which will be further discussed in {Sections \ref{inference}-\ref{experiment}}): As demonstrated in PC-SBL, the performance is insensitive to the specific value of \( \beta \) when a uniform value is applied across all coefficients. However, we find that the key to performance improvement lies in introducing \textit{relative differences} among \( \beta_i \), rather than merely tuning a single global \( \beta \) for all coefficients. Consequently, we assume that \( \beta_i \) follows an independent and identically distributed (i.i.d.) Gamma prior, i.e.,  
\begin{equation}
	p(\vec{\boldsymbol{\beta}})=\prod_{i=1}^{N-1}\text{Gamma}\left(\beta_i \mid c, d\right)= \prod_{i=1}^{N-1}\frac{d^c}{\Gamma(c)}\beta_i^{c-1} e^{-d\beta_i},\label{prior_beta}
\end{equation}
where the Gamma prior is conjugate and enforces the non-negativity of $\beta_i$. The structural patterns in the signals are captured primarily through the adaptive learning of $\beta_i$ within the SPP-SBL framework.


In summary, the overall structure of the Bayesian
hierarchical model is depicted in Fig. \ref{directed graph} (\textbf{B4}).

\subsection{Posterior Estimation}

Based on the observation model \eqref{basic CS}, the Gaussian likelihood is given by  
\begin{equation}
	p(\mathbf{y} \mid \mathbf{x}, \gamma)=\mathcal{N}(\boldsymbol{\Phi}\mathbf{x}, \gamma^{-1}\mathbf{I}),\label{likelihood}
\end{equation}  
where \( \gamma \) represents the inverse variance, which is typically assigned a Gamma prior, i.e., 
\begin{equation}
	p(\gamma)=\text{Gamma}(\gamma\mid g,h)=\frac{h^g}{\Gamma(g)}\gamma^{g-1} e^{-h\gamma}.\label{prior_gamma}
\end{equation}
With the prior \eqref{prior_x} and the likelihood \eqref{likelihood}, the posterior distribution of \( \mathbf{x} \) can be derived using Bayes' theorem as  
\begin{align}
	p(\mathbf{x} \mid \mathbf{y}, \boldsymbol{\alpha}, \vec{\boldsymbol{\beta}},\gamma) 
	&= \mathcal{N}(\boldsymbol{\mu}, \boldsymbol{\Sigma})\label{posterior_x}\\
	&\propto p(\mathbf{y} \mid \mathbf{x}, \gamma) p(\mathbf{x} \mid \boldsymbol{\alpha}, \vec{\boldsymbol{\beta}}),\notag
\end{align}  
where  
\begin{equation}
	\left\{
	\begin{array}{l}
		\boldsymbol{\mu} = \gamma\boldsymbol{\Sigma}\boldsymbol{\Phi}^T \mathbf{y}, \\  
		\boldsymbol{\Sigma} = \left(\gamma\boldsymbol{\Phi}^T\boldsymbol{\Phi}+\boldsymbol{\Lambda}\right)^{-1},\label{posterior}
	\end{array}
	\right.
\end{equation}  
with \( \boldsymbol{\Lambda} \triangleq \diag \left( \boldsymbol{T}_{\rm SPP} \boldsymbol{\alpha} \right) \). Thus, after estimating the hyperparameters \( \Theta = \{\boldsymbol{\alpha}, \vec{\boldsymbol{\beta}}, \gamma\} \), as described in Section \ref{inference}, the maximum a posteriori (MAP) estimate of \( \mathbf{x} \) corresponds to the mean of the posterior distribution, i.e.,  
\begin{equation}
	\mathbf{x}^{MAP}=\hat{\boldsymbol{\mu}}=\hat{\gamma}\hat{\boldsymbol{\Sigma}}\boldsymbol{\Phi}^T \mathbf{y}. \label{MAP}
\end{equation}

\section{Bayesian Inference: SPP-SBL Algorithm}\label{inference}

In this section, we employ Expectation Maximization (EM) formulation \cite{dempster1977maximum} to estimate the hyperparameters  \( \Theta = \{\boldsymbol{\alpha}, \vec{\boldsymbol{\beta}}, \gamma\} \) in posterior distribution.

Learning the hyperparameters essentially means finding the posterior mode that best matches the data, i.e., 
maximizing the posterior distribution \( p(\Theta\mid \mathbf{y}) \). Following the EM procedure, in the E-step, we treat \( \mathbf{x} \) as hidden variables and construct a lower bound of the log-posterior, i.e., the $Q$-function. Specifically, starting from the log-posterior:
\[
\log p(\Theta \mid \mathbf{y}) = \log \int p(\Theta, \mathbf{x} \mid \mathbf{y}) \, d\mathbf{x},
\]
we introduce the posterior of the latent variable $\mathbf{x}$ based on the previous estimate $\Theta^{(t-1)}$, i.e., $\mathbf{x} \sim p(\mathbf{x} \mid \mathbf{y}, \Theta^{(t-1)})$, and apply Jensen's inequality:
\begin{align}
	\log p(\Theta \mid \mathbf{y}) &= \log \int p(\Theta \mid \mathbf{x}, \mathbf{y}) p(\mathbf{x} \mid \mathbf{y}, \Theta^{(t-1)}) \, d\mathbf{x} \nonumber\\
	&\ge \int p(\mathbf{x} \mid \mathbf{y}, \Theta^{(t-1)}) \log p(\Theta \mid \mathbf{x}, \mathbf{y}) \, d\mathbf{x} \nonumber\\
	&= \mathbb{E}_{\mathbf{x} \mid \mathbf{y}, \Theta^{(t-1)}} [\log p(\Theta \mid \mathbf{x}, \mathbf{y})] \nonumber\\
	&\triangleq Q(\Theta \mid \Theta^{(t-1)}). \label{Q1}
\end{align}
Here, $Q(\Theta \mid \Theta^{(t-1)})$ is a lower bound of the log-posterior $\log p(\Theta \mid \mathbf{y})$, which also satisfies tangency condition. Maximizing it in the M-step ensures non-decreasing posterior values.
 To further simplify, it is observed that 
\begin{align}
	p(\Theta \mid \mathbf{x}, \mathbf{y})&\propto p(\boldsymbol{\alpha},\vec{\boldsymbol{\beta}} \mid \mathbf{x}) p(\gamma \mid \mathbf{y}, \mathbf{x}).\label{prop}
\end{align}
Therefore, \eqref{Q1} becomes 
\begin{align}
		Q(\Theta \mid \Theta^{(t-1)})=& \mathbb{E}_{\mathbf{x} \mid \mathbf{y}, \Theta^{(t-1)}} [\log p(\boldsymbol{\alpha},\vec{\boldsymbol{\beta}} \mid \mathbf{x})]\notag\\
		&+\mathbb{E}_{\mathbf{x} \mid \mathbf{y}, \Theta^{(t-1)}} [\log p(\gamma \mid \mathbf{y}, \mathbf{x})]+c,
\end{align}
where $c$ corresponds to the proportional constant  in \eqref{prop}. Thus, the \( Q \) function can be decomposed into two decoupled parts, which can be optimized separately. The first part is defined as \( Q(\boldsymbol{\alpha}, \vec{\boldsymbol{\beta}} \mid \boldsymbol{\alpha}^{(t-1)}, \vec{\boldsymbol{\beta}}^{(t-1)})\triangleq
\mathbb{E}_{\mathbf{x} \mid \mathbf{y}, \Theta^{(t-1)}} [\log p(\boldsymbol{\alpha},\vec{\boldsymbol{\beta}} \mid \mathbf{x})] \), and the second part is  \( Q(\gamma \mid \gamma^{(t-1)})\triangleq
\mathbb{E}_{\mathbf{x} \mid \mathbf{y}, \Theta^{(t-1)}} [\log p(\gamma \mid \mathbf{y}, \mathbf{x})] \). 
As for the first term, we have 
\begin{align}
	Q(\boldsymbol{\alpha}, \vec{\boldsymbol{\beta}} \mid \boldsymbol{\alpha}^{(t-1)}, \vec{\boldsymbol{\beta}}^{(t-1)})
	=&\mathbb{E}_{\mathbf{x} \mid \mathbf{y}, \Theta^{(t-1)}}\left[\log 
	p\left(\mathbf{x} \mid \boldsymbol{\alpha}, \vec{\boldsymbol{\beta}} \right)  \right]\notag\\
	+&\log p(\boldsymbol{\alpha})+\log p(\vec{\boldsymbol{\beta}})+c_1.\label{Q_split1}
\end{align}
Here, $c_1$ is a constant independent of $\boldsymbol{\alpha}, \vec{\boldsymbol{\beta}}$. Specifically, the expectation term in \eqref{Q_split1} can be expressed as
\allowdisplaybreaks
\begin{align}
	&\mathbb{E}_{\mathbf{x} \mid \mathbf{y}, \Theta^{(t-1)}}\left[\log 
	p\left(\mathbf{x} \mid \boldsymbol{\alpha}, \vec{\boldsymbol{\beta}} \right)  \right] \notag\\
	=& \frac{1}{2} \sum_{i=1}^{N} \bigg[ \log (\alpha_i + \beta_{i-1} \alpha_{i-1} + \beta_i \alpha_{i+1})  - (\alpha_i + \beta_{i-1} \alpha_{i-1}  \notag\\
	&+ \beta_i \alpha_{i+1}) 
	\int p(\mathbf{x}\mid \mathbf{y}, \boldsymbol{\alpha}^{(t-1)}, \vec{\boldsymbol{\beta}}^{(t-1)}) x_i^2 d \mathbf{x} \bigg].\label{simQ1}
\end{align}
According to \eqref{posterior_x}, the integral term above corresponds to the second-order moment of the \( i \)th component in a multivariate normal distribution, i.e.,
\begin{equation}
	\int p(\mathbf{x}\mid \mathbf{y}, \boldsymbol{\alpha}^{(t-1)}, \vec{\boldsymbol{\beta}}^{(t-1)}) x_i^2 d \mathbf{x} =\mathbb{E}_{\mathbf{x} \mid \mathbf{y}, \Theta^{(t-1)}}[x_i^2]= \hat{\mu}_i^2 + \hat{\Sigma}_{ii},\label{moment}
\end{equation}
where \( \hat{\mu}_i \) denotes the \( i \)th element of \( \hat{\boldsymbol{\mu}} \), and \( \hat{\Sigma}_{ii} \) represents the \( (i,i) \)th entry of \( \hat{\boldsymbol{\Sigma}} \). According to \eqref{prior_alpha}, \eqref{prior_beta}, \eqref{simQ1}, and \eqref{moment}, the first term of the $Q$ function in \eqref{Q_split1} becomes
\begin{align}
	&Q(\boldsymbol{\alpha}, \vec{\boldsymbol{\beta}} \mid \boldsymbol{\alpha}^{(t-1)}, \vec{\boldsymbol{\beta}}^{(t-1)})\notag \\
	=& \sum_{i=1}^{N} \bigg[ (a-1) \log \alpha_i - b \alpha_{i} + (c-1) \log \beta_{i} - d\beta_{i} \notag \\
	&\quad - \frac{1}{2} (\alpha_i + \beta_{i-1} \alpha_{i-1} + \beta_i \alpha_{i+1}) 
	\left(\hat{\mu}_i^2 + \hat{\Sigma}_{ii} \right) \notag \\
	&\quad + \frac{1}{2} \log (\alpha_i + \beta_{i-1} \alpha_{i-1} + \beta_i \alpha_{i+1})\bigg].\label{Q1 final}
\end{align}
For the second term of  $Q$ function, recalling (26), we obtain
\begin{align}
	Q(\gamma \mid \gamma^{(t-1)})=&\mathbb{E}_{\mathbf{x} \mid \mathbf{y}, \Theta^{(t-1)}} [\log p(\gamma \mid \mathbf{y}, \mathbf{x})]\notag\\
	=&\mathbb{E}_{\mathbf{x} \mid \mathbf{y}, \Theta^{(t-1)}}\left[\log p(\gamma)+\log p(\mathbf{y} \mid \mathbf{x},\gamma)\right]+c_2\notag\\
	=& -\frac{\gamma}{2} \mathbb{E}_{\mathbf{x} \mid \mathbf{y}, \Theta^{(t-1)}}\left[\vert| \mathbf{y}-\boldsymbol{\Phi} \mathbf{x}\vert|_2^2\right]+\frac{M}{2}\log \gamma \notag\\
	&~+g \log \gamma -h\gamma . \label{Q2}
\end{align}

In M-step, we separately maximize the above $Q$ functions in \eqref{Q1 final} and \eqref{Q2} to get the estimation of $\Theta$, i.e.,
\begin{align}
	&\boldsymbol{\alpha}^{(t)} = \arg \max_{\boldsymbol{\alpha}}~Q(\boldsymbol{\alpha}, \vec{\boldsymbol{\beta}}^{(t-1)} \mid \boldsymbol{\alpha}^{(t-1)}, \vec{\boldsymbol{\beta}}^{(t-1)}), \label{Max_alpha}\\
&	\vec{\boldsymbol{\beta}}^{(t)} = \arg \max_{\vec{\boldsymbol{\beta}}}~Q(\boldsymbol{\alpha}^{(t)}, \vec{\boldsymbol{\beta}} \mid \boldsymbol{\alpha}^{(t-1)}, \vec{\boldsymbol{\beta}}^{(t-1)}), \label{Max_beta}\\
& \gamma^{(t)} =  \arg \max_{\gamma}~Q(\gamma \mid \gamma^{(t-1)}).\label{Max_gamma}
\end{align}

\paragraph{Update $\boldsymbol{\alpha}$} 

To solve \eqref{Max_alpha}, we note that, unlike conventional  sparse Bayesian learning, where each hyperparameter can be updated independently due to the separable structure of the cost function, the hyperparameters in our case are coupled through the logarithmic term $\log (\alpha_i + \beta_{i-1} \alpha_{i-1} + \beta_i \alpha_{i+1})$, making analytical optimization intractable. Although gradient-based methods can be applied, they lack closed-form updates and incur higher computational complexity. To address these limitations, we follow the analysis in  \cite{fang2014pattern}, adopting an alternative strategy that seeks a simple, analytical sub-optimal solution by examining the optimality conditions, while ensuring monotonic improvement of the objective function during iterations. 

For brevity, we denote $Q(\boldsymbol{\alpha}, \vec{\boldsymbol{\beta}} \mid \boldsymbol{\alpha}^{(t-1)}, \vec{\boldsymbol{\beta}}^{(t-1)})$ by $Q\left(\boldsymbol{\alpha}, \vec{\boldsymbol{\beta}}\right)$ hereafter.  Let $\xi_i \triangleq  (\alpha_i+\beta_{i-1} \alpha_{i-1}+\beta_i\alpha_{i+1})^{-1}  $, and
\begin{align}
	\eta_i \triangleq &\left(\mu_i^2 + \Sigma_{ii} \right) + \beta_i \left(\mu_{i+1}^2 + \Sigma_{i+1,i+1} \right)  \notag \\
	&~ +\beta_{i-1} \left(\mu_{i-1}^2 + \Sigma_{i-1,i-1} \right), \label{update_eta}
\end{align}
the first-order derivative of \eqref{Q1 final} becomes
	\begin{align*}
	\frac{\partial Q\left(\boldsymbol{\alpha}, \vec{\boldsymbol{\beta}}\right)}{\partial \alpha_{i}}= \frac{a-1}{\alpha_i}-b+\frac{1}{2}\xi_i +\frac{\beta_i}{2}\xi_{i+1} +\frac{\beta_{i-1}}{2}\xi_{i-1}-\frac{1}{2}\eta_i.
\end{align*}
Suppose $\boldsymbol{\alpha}^*$
is a stationary point, satisfying the first-order optimality condition, i.e.,
\begin{equation}
	\left.\frac{\partial Q(\boldsymbol{\alpha}, \vec{\boldsymbol{\beta}})}{\partial \boldsymbol{\alpha}} \right|_{\boldsymbol{\alpha}=\boldsymbol{\alpha}^*}=\mathbf{0}.
\end{equation}
Then, for $\forall i=1,\cdots, N$, we have
\begin{align}
	& \frac{a-1}{\alpha_i^*}+\frac{1}{2}\left(\xi_i^*+\beta_{i} \xi_{i+1}^*+\beta_{i-1}\xi_{i-1}^*\right) = \frac{1}{2}\eta_i +b\label{big}.
\end{align}
Since $\left\{\alpha_i, \beta_i\right\}\ge 0 $,  it follows that $(\alpha_i^*)^{-1}>\xi_i^*,$ $ (\beta_{i-1}\alpha_{i-1}^*)^{-1}>\xi_i^*,$ and $ (\beta_{i}\alpha_{i+1}^*)^{-1}>\xi_i^*$. Therefore the left hand side of \eqref{big} can be bounded as
\begin{equation*}
	\frac{a-1}{\alpha_i^*}\le LHS \le 	\frac{a+0.5}{\alpha_i^*},
\end{equation*}
which concludes that 
\begin{equation*}
	\alpha_i^* \in \left[\frac{a-1}{b+0.5\eta_i}, \frac{a+0.5}{b+0.5\eta_i}\right].
\end{equation*}
We finally adopt the analytical solution to \eqref{Max_alpha} as
\begin{equation}
	\alpha_i = \frac{a+0.5}{b+0.5\eta_i}, ~\forall i=1,\cdots,N. \label{update_alpha}
\end{equation}
where \( a \) and \( b \) are typically set to small values (e.g., $10^{-4}$). 
It can be seen that the update form in \eqref{update_alpha} is structurally identical to that of classical SBL ((48) in \cite{tipping2001sparse}), with the only difference lying in the definition of \(\eta_i\). In classical SBL, \(\eta_i = \mu_i^2 + \Sigma_{ii}\), whereas in our formulation, \(\eta_i\) is defined in \eqref{update_eta} as a weighted combination of neighboring terms, incorporating \(\beta_i\) and \(\beta_{i+1}\). This difference underscores the critical role of learning \(\beta_i\) in capturing structured sparsity.

\vspace{2mm}

\paragraph{Update $\vec{\boldsymbol{\beta}}$}
To derive a closed-form update for $\vec{\boldsymbol{\beta}}$, we directly solve for the first-order stationary point of \eqref{Max_beta} as
\begin{align}
	\frac{\partial Q\left(\boldsymbol{\alpha}, \vec{\boldsymbol{\beta}}\right)}{\partial \beta_{i}}&=\frac{c-1}{\beta_i}-d +\frac{1}{2} \left(\alpha_{i+1}\xi_i + \alpha_{i}\xi_{i+1} \right)\notag\\
	&-\frac{1}{2}\left[\alpha_{i+1}\left(\hat{\mu}_i^2+\hat{\phi}_{ii}\right)+\alpha_i \left(\hat{\mu}_{i+1}^2+\hat{\phi}_{i+1,i+1}\right)\right]\label{cubic} 
\end{align}
Accordingly, the problem is transformed into finding the roots of the above equation in $\beta_i$. We first analyze the properties of the equation \eqref{cubic}, and present the following two theorems to facilitate root-finding.

\begin{theorem}\label{Thm1}
	When \( c > 1 \), the equation \eqref{cubic} for \( \beta_i \) ( \(\forall  i \)) exists only one positive real root.
\end{theorem}

\begin{proof}
	For notational simplicity, here we denote $A \triangleq \alpha_i+\beta_{i-1} \alpha_{i-1}$, $E \triangleq \alpha_{i+1}+\beta_{i+1}\alpha_{i+2}$ and $B \triangleq \frac{1}{2}\left[\alpha_{i+1}\left(\hat{\mu}_i^2+\hat{\phi}_{ii}\right)+\alpha_i \left(\hat{\mu}_{i+1}^2+\hat{\phi}_{i+1,i+1}\right)\right]$. Then, 
	\eqref{cubic} can be rewritten as
	\begin{align}
		&\frac{c-1}{\beta_i}-d+\frac{1}{2}\left[\frac{\alpha_{i+1}}{A+\beta_i\alpha_{i+1}}+\frac{\alpha_{i}}{E+\beta_i\alpha_{i}}\right]-B=0, \label{original cubic}\\
		\Leftrightarrow &~ \left[2(B+d)\alpha_i\alpha_{i+1}\right]\beta_{i}^3 +\left[2(B+d)(A\alpha_i+E\alpha_{i+1})\right. \notag\\
		&\left. -2c\alpha_i \alpha_{i+1}\right] \beta_{i}^2+\left[(1-2c)(A\alpha_i +E \alpha_{i+1})\right. \notag\\
		&\left. +2(B+d)AE\right]\beta_{i}+2(1-c)AE = 0 \label{simplify}
	\end{align}
	Let $\beta_i^{(1)}, \beta_i^{(2)}, \beta_i^{(3)}$ denote the three roots of \eqref{cubic}. Based on the relationship between the roots and coefficients of the cubic equation, we have, for $c>1$,
	\begin{equation}
		\beta_i^{(1)}\beta_i^{(2)}\beta_i^{(3)} = -\frac{(1-c)AE}{(B+d)\alpha_i\alpha_{i+1}}>0.
	\end{equation}
   This implies that \eqref{cubic} admits at least one positive real root (if there are two complex roots, their product must be a positive real value).
	
	To prove the uniqueness of this real root, we examine the monotonicity of  \eqref{original cubic}, i.e.,
	\begin{equation*}
		f(\beta_i) \triangleq \frac{c-1}{\beta_i}-d+\frac{1}{2}\left[\frac{\alpha_{i+1}}{A+\beta_i\alpha_{i+1}}+\frac{\alpha_{i}}{E+\beta_i\alpha_{i}}\right]-B.
	\end{equation*}
	  For \( \beta_i \geq 0 \), \( f(\beta_i) \) is strictly decreasing as \( \beta_i \) increases. Thus, the uniqueness of the positive real root is established.
\end{proof}

\begin{remark}
	The relationship between the coupling parameter $\beta_i$ and the neighboring precision parameters $\alpha_i$ and $\alpha_{i+1}$ can be characterized through the relationship between the roots and coefficients of a cubic equation, as shown in \eqref{original cubic} and \eqref{simplify}. It can be observed that the dependence of the coupling parameter on the variance-layer parameters is not governed by a simple monotonic heuristic, but is inherently nonlinear. In particular, the coupling parameter $\beta_i$ does not characterize the structure of the target signal merely based on their absolute values, but rather on their relative values, i.e., the learning of local fluctuations. When the elements lie within a nonzero block, the signal amplitudes vary across positions, leading to heterogeneous variance-layer parameters. Consequently, $\beta_i$, as a root of the corresponding cubic equation, also exhibits noticeable variability. In contrast, when the elements lie outside nonzero blocks, the variance-layer parameters associated with zero entries tend to take nearly identical values, and accordingly, $\beta_i$, as the root of the corresponding cubic equation, exhibits a high degree of consistency. Based on this relationship, the coupling parameters can be jointly learned with the variance-layer parameters $\alpha_i$, thereby effectively capturing the underlying structure of the target signal. The visualization results in Fig. \ref{fig:beta} further corroborate this observation.
\end{remark}

Theorem \ref{Thm1} ensures that solving \eqref{cubic} yields a unique positive real solution for \(\beta_i\). Therefore, we find the closed-form solution of a cubic equation and select the unique positive real root as the estimate for coupling parameters \(\vec{\boldsymbol{\beta}}\).

According to \eqref{simplify}, let $\tilde{a}=2(B+d)\alpha_i\alpha_{i+1}$, $\tilde{b} =2(B+d)(A\alpha_i+E\alpha_{i+1})-2c\alpha_i \alpha_{i+1}$, $\tilde{c} = 2(B+d)AE+(1-2c)(A\alpha_i +E \alpha_{i+1})$, $\tilde{d}=2(1-c)AE $. 
Then \eqref{simplify} can be rewritten as the standard cubic form:
\begin{equation*}
	\tilde{a} \beta_i^3 + \tilde{b} \beta_i^2 + \tilde{c} \beta_i + \tilde{d} = 0.
\end{equation*}
Dividing both sides by \(\tilde{a}\) and substituting:
\begin{equation}
\beta_i = \tilde{\beta}_i - \tilde{b}/(3\tilde{a}), \label{kadan}
\end{equation}
the equation can be transformed into the depressed cubic form:
\begin{equation*}
\tilde{\beta}_i^3 + p \tilde{\beta}_i + q = 0,
\end{equation*}
where $p=(3\tilde{a}\tilde{c}-\tilde{b}^2)/(3\tilde{a}^2)$, $q=(27\tilde{a}^2\tilde{d}-9\tilde{a}\tilde{b}\tilde{c}+2\tilde{b}^3)/(27\tilde{a}^3)$. Then, using Cardano's formula \cite{abramowitz1965handbook}, let $\Delta = \left(\frac{q}{2}\right)^2 + \left(\frac{p}{3}\right)^3$ and $\omega =\frac{1}{2}(-1+\sqrt{3}j)$. The three roots of $\tilde{\beta}_i$ are given by
	\begin{align}
		&\tilde{\beta}_i^{(1)}=\sqrt[3]{-\frac{q}{2}+\sqrt{\Delta}}+\sqrt[3]{-\frac{q}{2}-\sqrt{\Delta}}, \\
		&\tilde{\beta}_i^{(2)}=\omega \sqrt[3]{-\frac{q}{2}+\sqrt{\Delta}}+\omega^2 \sqrt[3]{-\frac{q}{2}-\sqrt{\Delta}} ,\\
	&	\tilde{\beta}_i^{(3)}=\omega^2 \sqrt[3]{-\frac{q}{2}+\sqrt{\Delta}}+\omega \sqrt[3]{-\frac{q}{2}-\sqrt{\Delta}}.
	\end{align}
According to \eqref{kadan}, three roots of $\beta_i$ are thus given by
\begin{equation}
\beta_i^{(1)} = \tilde{\beta}_i^{(1)}-\frac{\tilde{b}}{3\tilde{a}},~
\beta_i^{(2)} = \tilde{\beta}_i^{(2)}-\frac{\tilde{b}}{3\tilde{a}}, ~
\beta_i^{(3)} = \tilde{\beta}_i^{(3)}-\frac{\tilde{b}}{3\tilde{a}}. \label{posterior_beta}
\end{equation}
Thus, we select the unique positive real root from \eqref{posterior_beta} as the estimate for coupling parameters.
The following result establishes bounds on $\beta_i$.
\begin{theorem}\label{Thm2}
	The unique positive real root of  \eqref{cubic} satisfies 
	\begin{equation}
		0 < \beta_i < \frac{c}{d}, ~~~~\forall i=1,\cdots,N.
	\end{equation}
\end{theorem}
\begin{proof}
As \(\beta_i \to +\infty\), the left-hand side (LHS) of \eqref{simplify} tends to infinity. Since \(\beta_i\) is the unique positive real root, we examine the value of \eqref{simplify} at \(\beta_i = c/d\). If the resulting expression is positive, it implies that \(c/d\) serves as an upper bound for \(\beta_i\).  
	It's clear that substitute $c/d$ leads to
	\begin{align}
		\frac{1}{d^3}&\left\{2AEd^3+2\left[\left(BE+\frac{\alpha_i}{2}\right)A+\frac{E\alpha_{i+1}}{2}\right]cd^2 \right.\notag\\
		&\left.+2Bc^2(A\alpha_i+E\alpha_{i+1})d+2c^3\alpha_i \alpha_{i+1}B\right\}>0, \notag
	\end{align}
	which completes the proof.
\end{proof}

Therefore, Theorem \ref{Thm1} and \ref{Thm2} provide theoretical guarantee for  the behavior of $\vec{\boldsymbol{\beta}}$. In practice, \(c\) should satisfy \(c > 1\), while the ratio \(c/d\) determines an upper bound for \(\beta_i\).

\begin{remark}
It is important to note that although varying \(c\) and \(d\) may influence the absolute values of the estimated \(\beta_i\), the algorithm is generally insensitive to their specific choices—similar to the PC-SBL model, where the absolute magnitude of $\beta$ has little impact on performance. {As shown in Section \ref{visualize beta},} \(\beta_i\) tends to fluctuate around a mean within the theoretical bounds given by Theorem \ref{Thm2}. More importantly, it is the relative variations of \(\beta_i\)—rather than their absolute magnitudes—that are essential for capturing the structural differences among blocks and thus driving the performance improvement of the algorithm.
\end{remark}

\vspace{2mm}

\paragraph{Update $\gamma$}

To solve \eqref{Max_gamma}, we follow the same approach as in classical SBL \cite{tipping2001sparse}. The learning rule is given by
\begin{equation}
	\gamma =\frac{M+2g}{\vert| \mathbf{y}-\boldsymbol{\Phi} \boldsymbol{\mu}\vert|_2^2 + \tr (\boldsymbol{\Sigma}\boldsymbol{\Phi}^T \boldsymbol{\Phi})+2h}, \label{update gamma}
\end{equation}
where $\boldsymbol{\mu}$ and $\boldsymbol{\Sigma}$ are computed as in \eqref{posterior}, and both $g$ and $h$ are typically set to small positive values (e.g., $10^{-4}$).

In conclusion, the Space-Power Prior SBL (\textbf{SPP-SBL}) algorithm is summarized in Algorithm \ref{SPP-SBL}. 

\begin{algorithm}[t!] 
	\caption{SPP-SBL Algorithm} 		\label{SPP-SBL} 
		\begin{algorithmic}[1]
			\STATE {\bfseries Input:}{ Measurement matrix $\boldsymbol{\Phi}$, \text{measurement vector}  $\mathbf{y}$, initialized positive parameters:  $\boldsymbol{\alpha}, \vec{\boldsymbol{\beta}}$, noise's precision $\gamma$ and hyperparameters: $c>1, d$. }  	
			\STATE {\bfseries Output:}{ Posterior mean $\hat{\mathbf{x}}^{MAP}.$} 
			\REPEAT
			\STATE Update $\boldsymbol{\alpha}$ by \eqref{update_alpha}. 
			\IF{\rm $\alpha_i >$ threshold (e.g. $1e10$)} 
			\STATE Let $\alpha_i=1e10$.
			\COMMENT{\textcolor{gray}{\small Zero out small energy for efficiency.}}
			\ENDIF
			\STATE Update $\vec{\boldsymbol{\beta}}$ by selecting the unique positive real root from \eqref{posterior_beta}.
			\COMMENT{\textcolor{gray}{\small The numerical algorithm for solving the cubic equation \eqref{simplify} can also yield a positive real root efficiently.}}
			\STATE Update $\boldsymbol{\mu}$ and $\boldsymbol{\Sigma}$ by \eqref{posterior}.
			\STATE Update $\gamma$ using \eqref{update gamma}.
			\UNTIL{convergence criterion met}
			\STATE $\hat{\mathbf{x}}^{MAP}=\boldsymbol{\mu}$.
			\COMMENT{\textcolor{gray}{\small Use posterior mean as estimate.}}
		\end{algorithmic}
\end{algorithm}

\begin{remark}[Computational complexity and practical efficiency]
	The proposed SPP-SBL has a per-iteration computational complexity of
	$\mathcal{O}(M^2N),$ which is identical to that of PC-SBL, A-EBSBL, and the original SBL \cite{ji2008bayesian}. A detailed per-iteration complexity analysis is provided in Appendix~\ref{APP: complexity}. Although SPP-SBL introduces an additional update for the coupling parameters $\vec{\boldsymbol{\beta}}$, this step has only linear complexity in $N$ and is negligible compared to the dominant posterior mean and covariance updates shared by most SBL-type algorithms.
	Moreover, by jointly learning $\boldsymbol{\alpha}$ and $\vec{\boldsymbol{\beta}}$, SPP-SBL exhibits faster empirical convergence, achieving comparable reconstruction accuracy with significantly fewer iterations, and consequently improved practical runtime efficiency, as shown in Fig. \ref{time}.
\end{remark}

\section{Discussion} \label{sec: discussion}

Up to now, SPP-SBL has effectively addressed the open problem of estimating the coupling parameter $\vec{\boldsymbol{\beta}}$ in pattern-coupled Bayesian algorithms. As mentioned in Section \ref{intro}, the issues of boundary overestimation \cite{dai2018non, wang2018alternative, wang2020structured} and limitations in adaptability to more complex data structures, such as multi-pattern datasets \cite{sant2022block} and chain-structured datasets \cite{korki2016iterative}, arising from the coupling parameter estimation problem, can be resolved through accurate estimation of the coupling parameter. Therefore, in this section, we will discuss these two issues in detail and analyze the reasons why the SPP-SBL model effectively addresses them.

\subsection{Boundary Overestimation}

The boundary overestimation issue was first discussed in A-EBSBL \cite{wang2018alternative} and arises in classical pattern-coupled sparse Bayesian methods, such as PC-SBL and A-EBSBL. Intuitively, this problem occurs when a zero element lies at the boundary of a non-zero block.
Taking  PC-SBL as example (same applies to A-EBSBL, etc.), when element \( x_i \) lies within a non-zero cluster, i.e., both \( x_{i-1} \) and \( x_{i+1} \) are non-zero, the rule in \cite{fang2014pattern} (equations (21) and (26)) generates a non-zero \( \alpha_i \), indicating a non-zero \( x_i \). However, when \( x_{i} \) lies at the boundary of a block, it leads to the problem of boundary overestimation. Fig. \ref{bound} illustrates all possible boundary cases for $x_i$, where existing algorithms may encounter overestimation issues. Take Fig. \ref{bound}(a) as example: consider the case where \( x_{i-1} \) is non-zero, and both \( x_i \) and \( x_{i+1} \) are zero. Therefore, \( \alpha_i \) calculated according to \cite{fang2014pattern} (equations (21) and (26)) would also be non-zero, which consequently results in \( x_i \) being non-zero. This issue is also summarized in \cite{wang2018alternative} as a limitation caused by the structure prior. 

	\begin{figure}[!h]
	\begin{center}
		\includegraphics[width=2.7in]{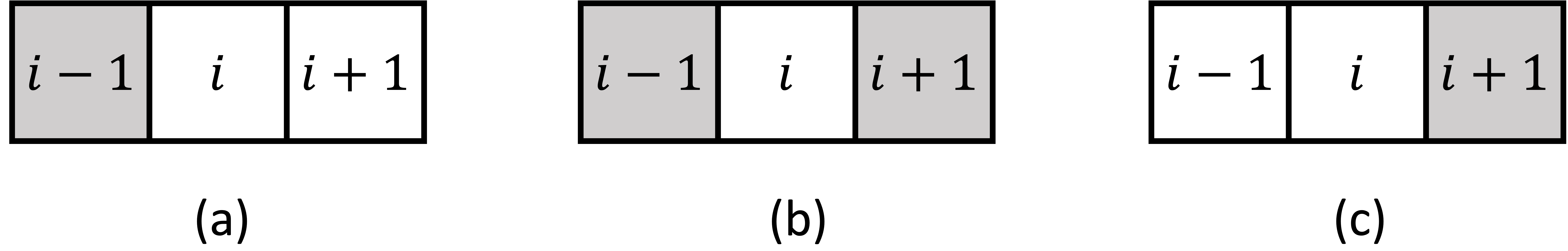}
		\vspace{-1mm}
	\end{center}
	\caption{
		The three possible boundary cases for $x_i$, where the shaded regions indicate nonzero positions and the white regions represent zero positions.}
	\label{bound}
\end{figure}

However, we argue that the root cause of this problem lies fundamentally in the estimation of coupling parameter. In previous works, the same coupling parameter \( \beta \) was used between adjacent elements, which made it impossible to adaptively determine the structure at the boundaries through the learning of \( \beta \). 
To address this, our SPP-SBL model learns \( \vec{\boldsymbol{\beta}} \in \mathbb{R}^{N-1} \) in a manner that assigns weights to neighboring dependencies (as shown in \eqref{update_alpha} and \eqref{update_eta}), enabling the adaptive identification of boundaries. We will further demonstrate the improvement of our algorithm on boundary learning in Section \ref{experiment}.

\subsection{Complex Block Sparse Structure} \label{subsec:complex}

As mentioned in  \cite{sant2022block}, previous works have stated that algorithms using rigid SBL hyperparameter coupling, while effective for flexible block-sparse recovery, struggle to recover signals containing both block-sparse and isolated sparse components, which we refer to as \textit{Multi-Pattern Data}. They argued these stronger priors are biased towards block structures in the underlying signal. We have also observed that this issue extends beyond multi-pattern data, arising similarly in signals with temporal dependencies, such as \textit{Chain Data}. Thus, handling complex block-sparse signals remains a significant challenge for traditional SBL hyperparameter coupling methods.

We also argue that the main cause of this challenge lies in the use of a fixed coupling parameter \( \beta \), which not only remains the same across all elements but also lacks the ability to adaptively learn from the data. This limitation prevents the model from flexibly adjusting the dependency between neighboring elements, making it difficult to obtain reliable results when the data contains both clustered and isolated components. For temporally dependent signals, such as chain data, it is intuitive that each element leverages information differently from its previous and next neighbors—something that traditional models fail to account for.

In contrast, the SPP-SBL model learns \( \vec{\boldsymbol{\beta}} \), allowing it to adaptively shrink toward the boundaries of the data and adjust the dependencies between neighboring elements. In Section \ref{experiment}, we will demonstrate that the algorithm can effectively capture block sparse patterns, even in the case of chain data or when there are significant scale differences between isolated values and blocks, by learning the underlying pattern of \( \vec{\boldsymbol{\beta}} \).

\section{Numerical Experiments}\label{experiment}

In this section, we compare the proposed SPP-SBL algorithm\footnote{Matlab codes for our algorithm are available at \url{https://github.com/YanhaoZhang1/SPP-SBL}.} against twelve existing methods, categorized as follows:  
\begin{itemize}
	\item Sparse learning (without structural information): (1) SBL \cite{tipping2001sparse}.
	\item Predefined block partition: (2) Group Lasso \cite{yuan2006model}, (3) Group BPDN \cite{van2011sparse}, (4) Model-based CS (Mb-CS) \cite{baraniuk2010model}, (5) BSBL \cite{zhang2013extension}, and (6) DivSBL \cite{NEURIPS2024_ead542f1}.
	\item Pattern design: (7) PC-SBL \cite{fang2014pattern}, (8) A-EBSBL \cite{wang2018alternative}, (9) RCS-SBL \cite{wang2020structured} and (10) StructOMP \cite{huang2009learning}.
	\item Tailored for chain-type block data: (11) Block-IBA \cite{korki2016iterative}.
	\item Differential variation regularization: (12)TV-SBL \cite{sant2022block}. 
\end{itemize}

The design matrix \( \boldsymbol{\Phi }\) is randomly generated with i.i.d. Gaussian distribution. The signal-to-noise ratio (SNR) is defined as  
$
\text{SNR (dB)} = 20 \log_{10}\left( \frac{\| \boldsymbol{\Phi } \mathbf{x}\|_2}{\|\mathbf{n}\|_2} \right),
$
and is set to 15 dB unless otherwise specified. 
For the hyperparameters $c$ and $d$ in the prior distribution of the coupling parameter $\vec{\boldsymbol{\beta}}$, we set $c=10$ and $d=1$ for convenience. As shown in Section~\ref{subsec:not sensitive} and Appendix~\ref{App:init}, the experimental results are largely insensitive to the choice of these hyperparameter initializations.

All results are averaged over 500 independent random trials. 
Let $\hat{\mathbf{x}}$ 
denotes the estimate of the true signal $\mathbf{x}_{\text{true}}$. The performance is evaluated using the following three metrics:
\begin{enumerate}
	\item[1.] Estimation Accuracy: Normalized Mean Squared Error (NMSE), defined as  
	\begin{equation*}
		\text{NMSE} = \frac{\|\hat{\mathbf{x}} - \mathbf{x}_{\text{true}}\|_2^2}{\|\mathbf{x}_{\text{true}}\|_2^2}.
	\end{equation*}
	\item [2.] Support Similarity: Correlation (Corr), or cosine similarity, measured as
		\begin{equation*}
		\text{Corr} = \frac{\hat{\mathbf{x}}^T \mathbf{x}_{\text{true}}}{\|\hat{\mathbf{x}}\|_2 \|\mathbf{x}_{\text{true}}\|_2}.
	\end{equation*}
	\item [3.] Support Accuracy: As introduced in \cite{wang2018alternative}\cite{sant2022block},  $\text{supp}(\hat{\mathbf{x}})$ denotes the recovered support of $\hat{\mathbf{x}}$. Then, the support recovery rate (SRR) is defined as
	\begin{equation*}
		\text{SRR} =\frac{\vert \text{supp}(\hat{\mathbf{x}})\cap \text{supp}(\mathbf{x}_{\text{true}})\vert }{\vert \text{supp}(\hat{\mathbf{x}})- \text{supp}(\mathbf{x}_{\text{true}})\vert+\vert \text{supp}(\mathbf{x}_{\text{true}})\vert },
	\end{equation*}
	It is observed that when the estimated support set exactly matches the true support set, $\text{SRR}= 1$.
\end{enumerate}

\subsection{Heteroscedastic Block Sparse Signal}\label{syn data}

\begin{table}[!t]
	\centering
	\caption{Performance comparison on heteroscedastic synthetic data. Our method is highlighted in \colorbox{customblue}{blue}, and the best results are shown in \textbf{bold}. }
	\label{tab:heteroscedastic}
	\begin{tabularx}{\linewidth}{lcccc}
		\toprule
		\multicolumn{4}{c}{\textbf{Heteroscedastic Block Sparse Signal}}\\
		\midrule
		\textbf{Algorithm} & \textbf{NMSE (std)} & \textbf{Corr (std)} & \textbf{SRR (std)} \\
		\midrule
		BSBL          & 0.1094 (0.0707) & 0.9457 (0.0376) & 0.7028 (0.1345) \\
		PC-SBL        & 0.0994 (0.0425) & 0.9528 (0.0192) & 0.7200 (0.1212) \\
		Block-IBA     & 0.1581 (0.1062) & 0.9166 (0.0621)& 0.6317 (0.1486) \\
		A-EBSBL       & 0.0952 (0.0306) & 0.9533 (0.0147) & 0.7358 (0.1303) \\
		SBL           & 0.1538 (0.1058) & 0.9261 (0.0508)& 0.6473 (0.1526)  \\
		Group Lasso   & 0.1486 (0.0546) & 0.9246 (0.0293) & 0.6756 (0.1172) \\
		Group BPDN    & 0.1522 (0.0551) & 0.9223 (0.0296)& 0.6736 (0.1178)  \\
		TV-SBL   & 0.1217 (0.0650) & 0.9387 (0.0341) & 0.6731 (0.1447) \\
		StructOMP     & 0.1360 (0.0679) & 0.9393 (0.0292) & 0.6665 (0.1273)\\
		Mb-CS  &  0.1766 (0.1124) & 0.9119 (0.0561) & 0.6193 (0.1271)\\
		RCS-SBL  &  0.2048 (0.2698) & 0.9128 (0.1057) & 0.6363 (0.2588)\\
		DivSBL        & 0.0640 (0.0367) & 0.9684 (0.0185)& 0.7758 (0.1150) \\
		\rowcolor{customblue}
		SPP-SBL       & \textbf{0.0402} (\textbf{0.0190}) & \textbf{0.9801} (\textbf{0.0095}) & \textbf{0.8151} (\textbf{0.0977}) \\
		\bottomrule
	\end{tabularx}
\end{table}

 We first evaluate the algorithms on the heteroscedastic synthetic data introduced in \cite{NEURIPS2024_ead542f1}, where the block sizes, locations, nonzero magnitudes, and signal variances are randomly generated to mimic real-world data patterns.  
Specifically, we generate sparse signals with a dimensionality of \(N = 162\), containing \(K = 40\) nonzero entries randomly distributed across four blocks with randomly assigned sizes and locations, and set the measurement rate to $0.5$.

The reconstruction results are summarized in Table~\ref{tab:heteroscedastic}. 
Notably, SPP-SBL achieves statistically significant improvements across all evaluation metrics: it attains the lowest NMSE (improving by approximately 37\% over the second-best and 57\% over the third-best methods), the highest Pearson correlation coefficient, and the optimal support recovery rate. 
Moreover, the minimized standard deviations further substantiate the robustness of SPP-SBL.

\subsection{The Effectiveness of Learning the Coupling Parameter $\vec{\boldsymbol{\beta}}$}\label{subsec:not sensitive}

In this experiment, we demonstrate the improvement in recovery performance brought by coupling parameter learning. Following the dataset setup described in Section \ref{syn data}, we evaluate the success rate across measurement ratios ranging from 0.3 to 0.9 in increments of 0.05, under an SNR of $50$ dB. The success rate is defined as the proportion of successful recoveries to the total number of trials, where a recovery is considered successful if the NMSE is no greater than $10^{-5}$.

The comparison algorithms include the PC-SBL model with a fixed coupling parameter $\beta$, where $\beta$ is set to 0, 0.5, 1, and 5, respectively. According to Theorem \ref{Thm2}, the absolute value of $\beta_i$ is governed by the ratio $c/d$. Here, we normalize $d=1$ and examine the effect of varying $c = 2$, $5$, and $10$ on the algorithm's performance. Fig.~\ref{not sensitive} illustrates the success rates of signal reconstruction across different measurement ratios.
We can find that the choice of hyperprior parameters has minimal impact on the performance of SPP-SBL algorithm, demonstrating strong robustness. Similarly, the selection of different positive values for $\beta$ has little effect on the performance of PC-SBL algorithm (consistent with the experimental results presented in Figure 2 of the original PC-SBL paper \cite{fang2014pattern}). 

	\begin{figure}[!t]
	\begin{center}
		\includegraphics[width=2.6in]{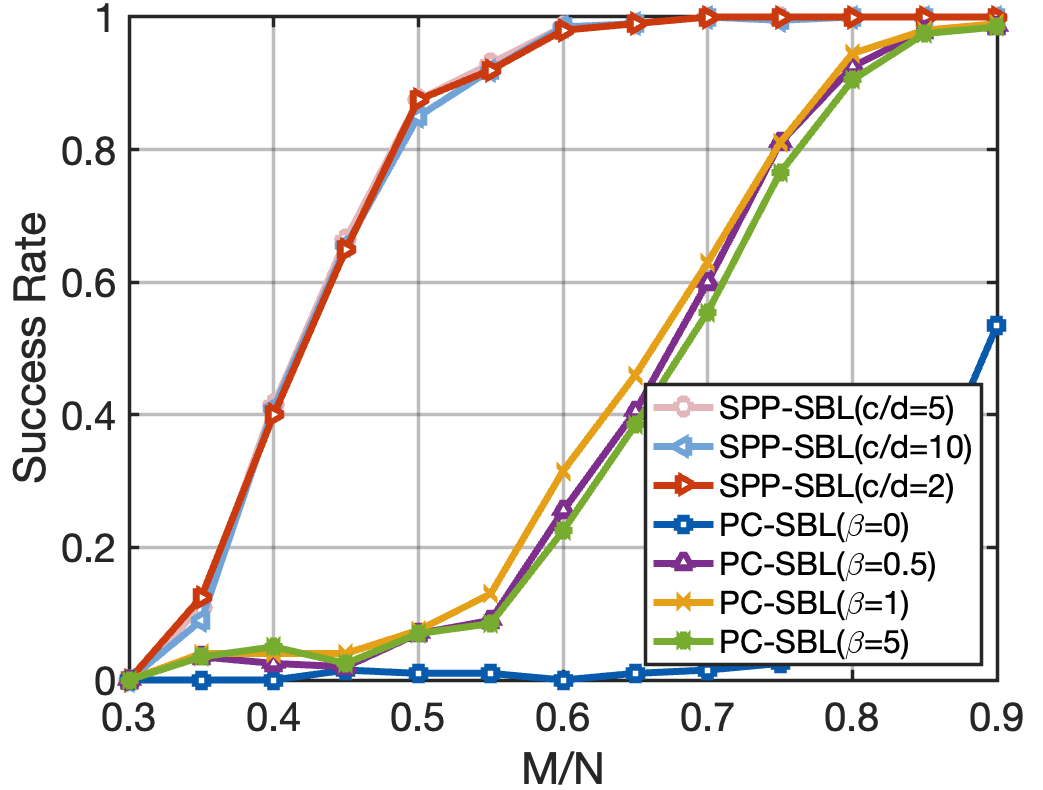}
	\end{center}
	\caption{The impact of different hyperprior parameters $c$ in SPP-SBL and different $\beta$ choices in PC-SBL on success rate.}
	\label{not sensitive}
\end{figure}

These results indicate that for models using a shared coupling parameter $\beta$, the performance gain from learning $\beta$ is limited. In contrast, notable  improvements are enabled by learning the relative magnitudes of a diversified coupling parameter vector $\vec{\boldsymbol{\beta}} \in \mathbb{R}^{N-1}$, which captures dependencies between adjacent entries and enhances recovery accuracy.

\subsection{Multi-pattern Sparse Signal} \label{multi}

As discussed in Section \ref{subsec:complex}, despite its rigid coupling structure, the SPP-SBL model can still adapt to complex data by learning the coupling parameters $\vec{\boldsymbol{\beta}}$. To evaluate this capability, we conduct experiments using datasets derived from the signal models proposed in \cite{sant2021general, sant2022block}, where the so-called multi-pattern sparsity (also referred to as hybrid sparsity) was introduced. Specifically, the full signal dimension is set to \( N = 162 \), the number of measurements to \( M = 80 \), and the total number of nonzero entries to \( K = 30 \). Among these, 25 entries are grouped into 3 randomly generated clusters, while the remaining 5 are isolated nonzero elements. This setup no longer conforms to the standard block sparse assumption and thus poses a more challenging scenario for block sparse recovery algorithms.

The experimental results are presented in Table \ref{tab:multi_pattern}. We can find that SPP-SBL outperforms all other algorithms across all three evaluation metrics. Its smaller standard deviation further indicates strong robustness. These results reinforce the fact that SPP-SBL can effectively adapt to data with diverse sparsity patterns by learning the coupling parameters. Importantly, this flexibility does not come at the expense of its performance in recovering block-sparse signals.

\subsection{Chain Data} \label{chain}

\begin{table}[!t]
	\centering
	\caption{Performance comparison on multi-pattern sparse signal recovery. Our method is highlighted in \colorbox{customblue}{blue}, and the best results are shown in \textbf{bold}.}
	\label{tab:multi_pattern}
	\begin{tabularx}{\linewidth}{lccc}
		\toprule
		\multicolumn{4}{c}{\textbf{Multi-pattern Sparse Signal}}\\
		\midrule
		\textbf{Algorithm} & \textbf{NMSE (std)} & \textbf{Corr (std)} & \textbf{SRR (std)} \\
		\midrule
		BSBL         & 0.1929 (0.0698) & 0.9015 (0.0406) &  0.4439 (0.0709)  \\
		PC-SBL       & 0.1601 (0.0532) & 0.9229 (0.0264) &  0.4006 (0.0617)  \\
		Block-IBA    & 0.1386 (0.0728) & 0.9283 (0.0405) &  0.5068 (0.0901) \\
		A-EBSBL      & 0.1317 (0.0394) & 0.9340 (0.0206) &  0.3884 (0.0491)  \\
		SBL          & 0.1801 (0.0849) & 0.9142 (0.0403) &  0.3969 (0.0683)  \\
		Group Lasso  & 0.2336 (0.0685) & 0.8758 (0.0399) &  0.3026 (0.0411)  \\
		Group BPDN   & 0.2386 (0.0677) & 0.8726 (0.0397) &  0.2926 (\textbf{0.0382})  \\
		TV-SBL    & 0.1327 (0.0606) & 0.9330 (0.0321) &  0.4659 (0.0769)  \\
		StructOMP    & 0.2732 (0.2183) & 0.8712 (0.1022) &  0.4567 (0.0973)  \\
		Mb-CS  &  0.1497 (0.0856) & 0.9257 (0.0451) & 0.5534 (0.0817)\\
		RCS-SBL  &  0.1584 (0.0493) & 0.9232 (0.0248) & 0.3960 (0.0509)\\
		DivSBL       & 0.0994 (0.0528) & 0.9504 (0.0279) &  0.5704 (0.0972)  \\
		\rowcolor{customblue}
		SPP-SBL      & \textbf{0.0670} (\textbf{0.0251}) & \textbf{0.9666} (\textbf{0.0129}) & \textbf{0.6197} (0.0775)  \\
		\bottomrule
	\end{tabularx}
\end{table}

\begin{table}[!t]
	\centering
	\caption{Performance comparison on chain-type block sparse signal recovery. Our method is highlighted in \colorbox{customblue}{blue}, and the best results are shown in \textbf{bold}.}
	\label{tab:blockIBA-data}
	\begin{tabularx}{\linewidth}{lccc}
		\toprule
		\multicolumn{4}{c}{\textbf{Chain-type Block Sparse Signal}}\\
		\midrule
		\textbf{Algorithm} & \textbf{NMSE (std)} & \textbf{Corr (std)} & \textbf{SRR (std)} \\
		\midrule
		BSBL         & 0.1417 (0.0291) & 0.9291 (0.0159) &  0.4905 (0.0817)  \\
		PC-SBL       & 0.0833 (0.0250) & 0.9588 (0.0126) &  0.5525 (0.0460)  \\
		Block-IBA    & 0.1001 (0.0508) & 0.9463 (0.0277) &  0.2752 (0.0485) \\
		A-EBSBL      & 0.3576 (0.0336) & 0.8060 (0.0199) &  0.1500 (\textbf{0.0078})  \\
		SBL          & 0.2861 (0.0783) & 0.8614 (0.0369) &  0.2238 (0.0426)  \\
		Group Lasso  & 0.1674 (0.0407) & 0.9200 (0.0208) &  0.3049 (0.0338)  \\
		Group BPDN   & 0.1732 (0.0411) & 0.9159 (0.0212) &  0.2852 (0.0297)  \\
		TV-SBL    & 0.1685 (0.0367) & 0.9124 (0.0202) &  0.2876 (0.0460)  \\
		StructOMP    & 0.1836 (0.0719) & 0.9076 (0.0374) &  0.5034 (0.0938)  \\
		Mb-CS  &  0.2332 (0.0520) & 0.8776 (0.0296) & 0.2354 (0.0288)\\
		RCS-SBL  &  0.0696 (0.0280) & 0.9652 (0.0145) & 0.5651 (0.0522)\\
		DivSBL       & 0.1307 (0.0496) & 0.9333 (0.0260) &  0.5704 (0.0739)  \\
		\rowcolor{customblue}
		SPP-SBL      & \textbf{0.0442} (\textbf{0.0163}) & \textbf{0.9779} (\textbf{0.0084}) & \textbf{0.7093} (0.0599)  \\
		\bottomrule
	\end{tabularx}
\end{table}

Following the setup in \cite{korki2016iterative}, we simulate block-sparse signals where each entry of $\mathbf{x}$ is generated as $x_i = s_i \theta_i$, with $s_i \in \{0,1\}$ indicating the support and $\theta_i$ denoting the magnitude. In vector form, this can be written as $\mathbf{x} = \mathbf{S}\boldsymbol{\theta}$, where $\mathbf{S} = \operatorname{diag}(\mathbf{s})$. The support vector $\mathbf{s}$ is modeled as a stationary first-order Markov process characterized by transition probabilities $p_{10} = \Pr(s_{i+1}=1 \mid s_i=0)$ and $p_{01} = \Pr(s_{i+1}=0 \mid s_i=1)$. The steady-state probabilities are given by $\Pr(s_i=0) = p = \frac{p_{01}}{p_{10}+p_{01}}$ and $\Pr(s_i=1) = 1-p$, with the parameters $p$ and $p_{10}$ fully specifying the process and $p_{01}$ computed as $p_{01} = \frac{p \cdot p_{10}}{1-p}$.
Under this generative model, the average block length of nonzero entries is $1/p_{01}$. 

In this experiment, the full signal dimension is set to $N=512$, with a relatively challenging observation setting of $M=130$ (approximately one-fourth of the full dimension). The Markov process parameters are configured as $p=0.8$ and $p_{10}=0.01$, resulting in an expected block size of approximately $25$ for the block-sparse signal. 
The reconstruction results are summarized in Table~\ref{tab:blockIBA-data}.

It is observed that on the chain-structured data, SPP-SBL still achieves significant improvements across all three metrics. Specifically, its NMSE is 47\% lower than the second-best method and 56\% lower than the third-best. Moreover, in terms of support recovery rate (SRR), while the average SRR of other methods remains below 60\%, SPP-SBL surpasses 70\%, demonstrating a substantial advantage. 

These results indicate that although SPP-SBL is not specifically designed for chain-structured data, it can still effectively handle such signals by learning the coupling parameters, exhibiting strong robustness across different structured scenarios.

To provide readers with a clear perspective on the performance–complexity tradeoff across different methods, we summarize the runtimes corresponding to the experiments in Sections~\ref{syn data} and \ref{multi}–\ref{chain} in Table~\ref{tab:runtime_all}. Among the pattern-based methods, such as PC-SBL, Block-IBA, A-EBSBL, and RCS-SBL, our proposed SPP-SBL algorithm demonstrates competitive and generally superior runtime performance. While it is naturally slower than heuristic methods like StructOMP and Mb-CS, SPP-SBL achieves significantly higher recovery accuracy with up to 76\% and 81\% improvements, demonstrating a favorable performance–complexity tradeoff.

\begin{figure*}[!t]
	\centering
	\begin{subfigure}[b]{0.325\textwidth}
		\includegraphics[width=\linewidth]{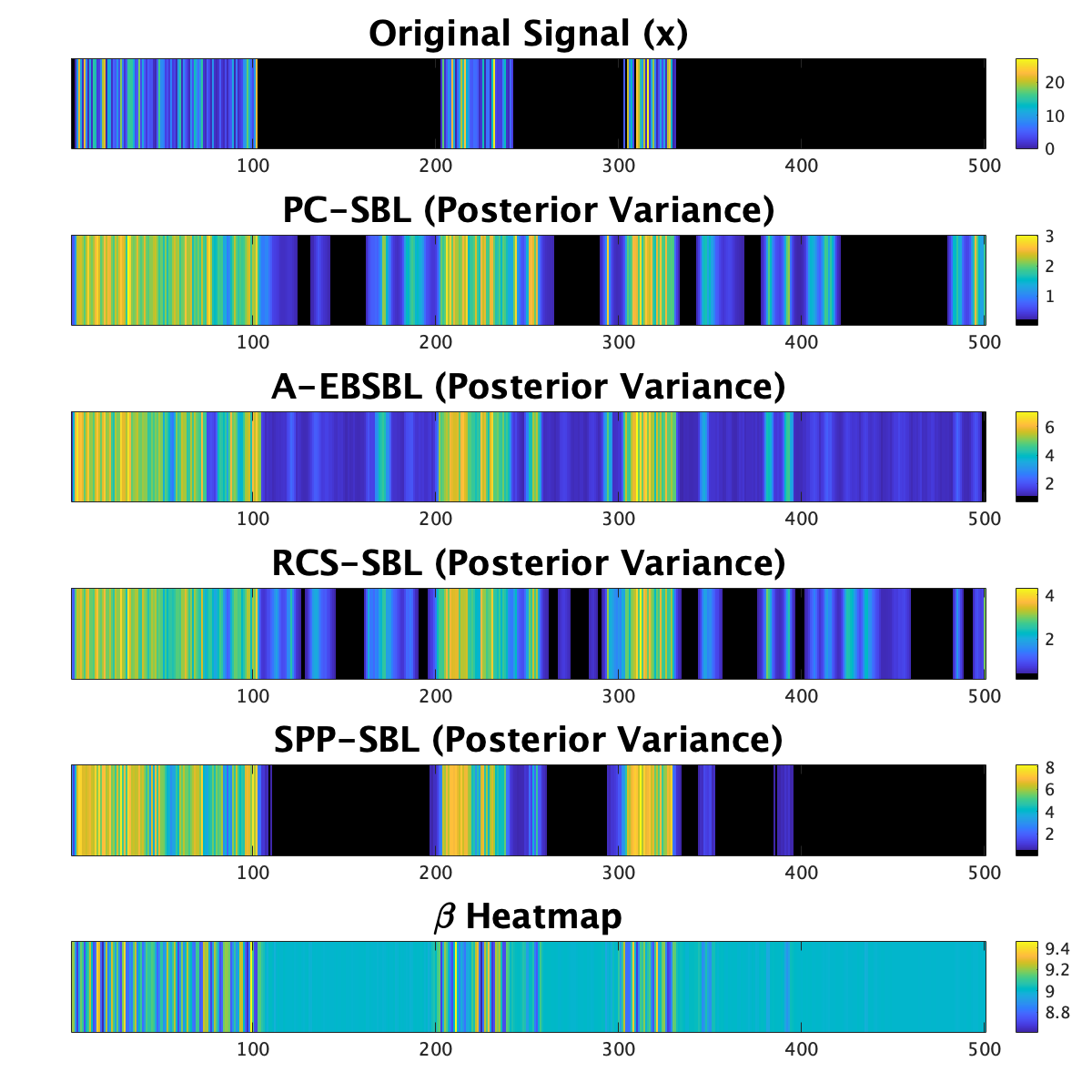}
		\caption{Heteroscedastic signal}
		\label{fig:subfig1}
	\end{subfigure}
	\hfill
	\begin{subfigure}[b]{0.325\textwidth}
		\includegraphics[width=\linewidth]{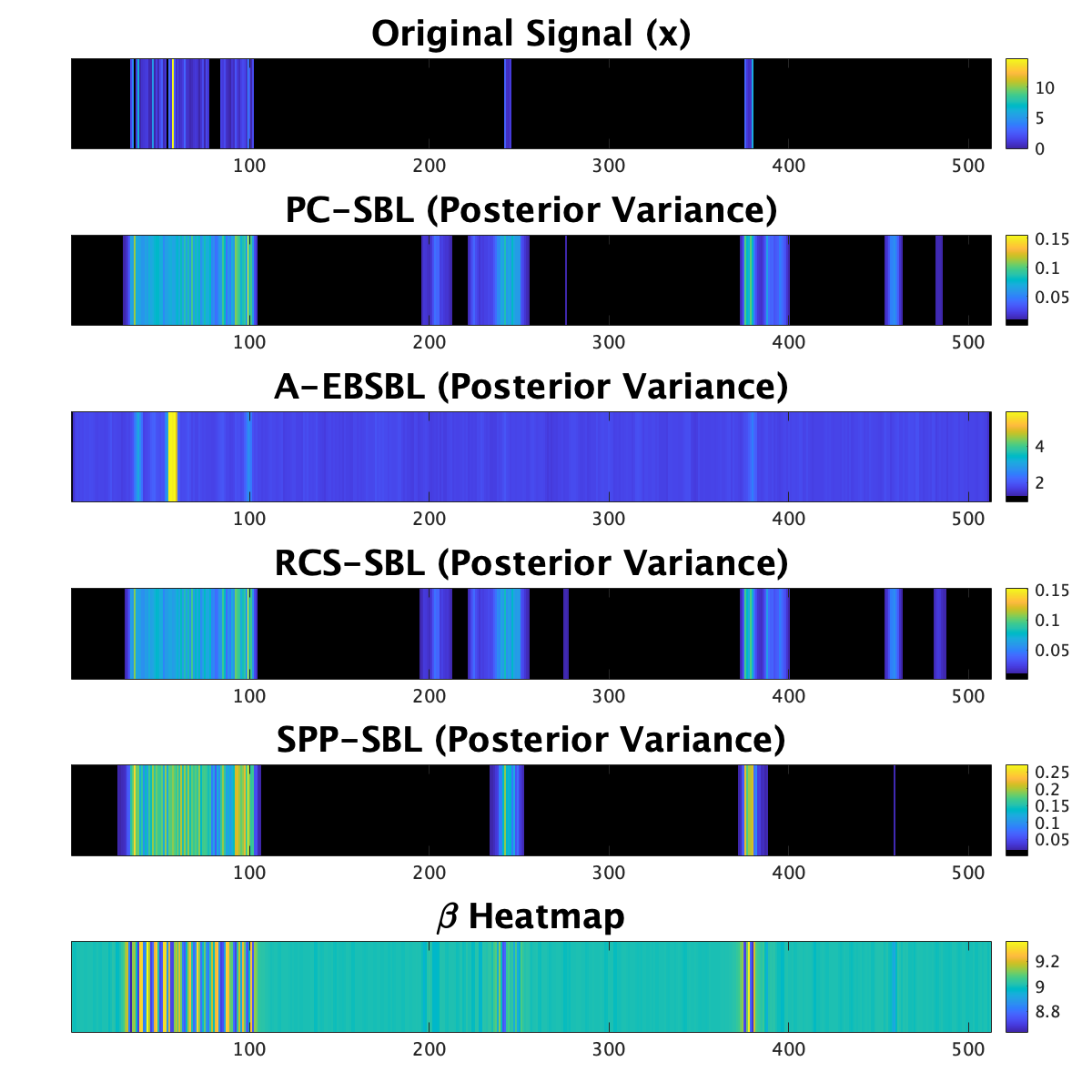}
		\caption{Chain signal}
		\label{fig:subfig2}
	\end{subfigure}
	\hfill
	\begin{subfigure}[b]{0.325\textwidth}
		\includegraphics[width=\linewidth]{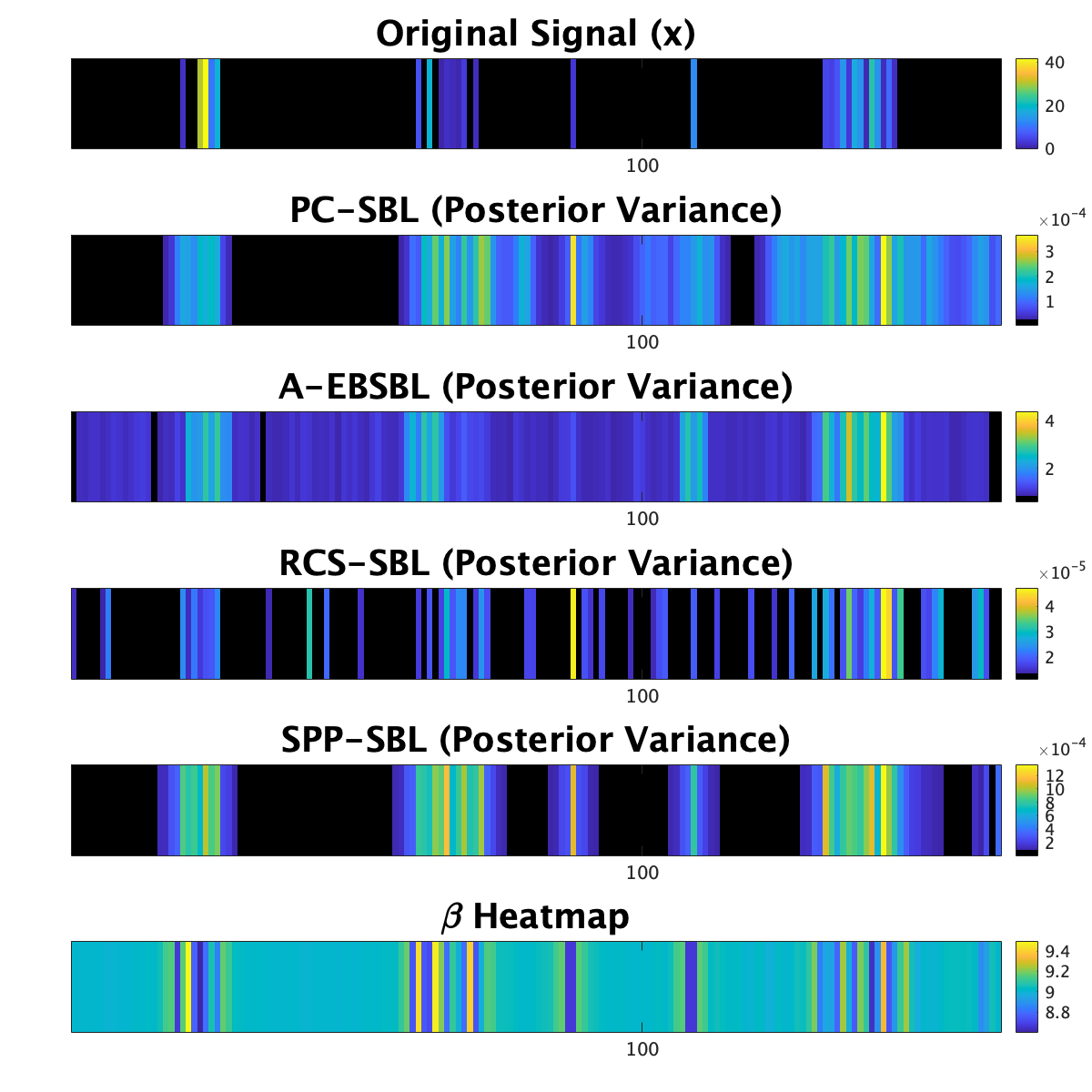}
		\caption{Multi-pattern signal}
		\label{fig:subfig3}
	\end{subfigure}
	\caption{Heatmaps of the posterior variance and the $\boldsymbol{\beta}$ values estimated by the SPP-SBL algorithm across three datasets.}
	\label{fig:beta}
\end{figure*}

\begin{figure*}[t]
	\centering
	\includegraphics[width=1\linewidth]{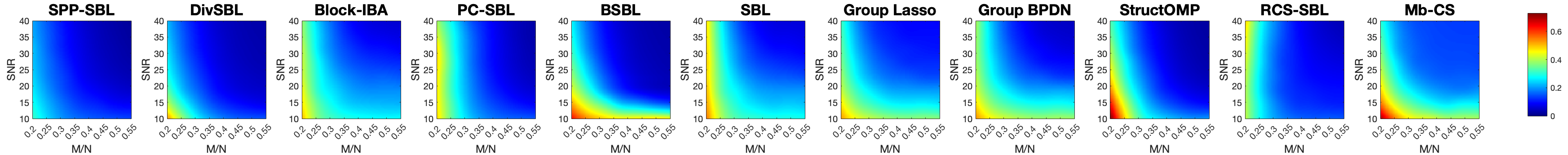}
	\caption{Phase transition  diagram under different  SNR and measurement ratios. \protect\footnotemark}
	\label{phase}
	
\end{figure*}
\footnotetext{TV-SBL and A-EBSBL failed on this audiosets, so only the results of the remaining nine algorithms are shown.}

\begin{table}[!t]
	\centering
	\caption{Runtime comparison across three types of block sparse data. 
			Our method is highlighted in \colorbox{customblue}{blue}.}
	\label{tab:runtime_all}
	\begin{tabularx}{\linewidth}{lccc}
		\toprule
		\multicolumn{4}{c}{~~~~~~~~~~~~~~\textbf{Runtime (s) (std)}} \\ 
		\midrule
		\textbf{Algorithm} &
		\textbf{Heteroscedastic} &
		\textbf{Multi-pattern} &
		\textbf{Chain-type} \\
		\midrule
		BSBL        & 0.0463 (0.0263)& 0.0608 (0.0295) &  0.2221 (0.0421)\\
		PC-SBL      & 0.1687 (0.0148)&  0.1661 (0.0139)& 1.9927 (0.2292)  \\
		Block-IBA   & 0.1845 (0.0873)     & 0.0925 (0.0234) & 0.7689 (0.3844) \\
		A-EBSBL     & 0.1530 (0.0367)&0.1369 (0.0400)  &  0.6186 (0.1601)\\
		SBL         & 0.0166 (0.0068)& 0.0100 (0.0066) & 0.0170 (0.0086) \\
		Group Lasso & 0.4800 (0.0499)& 0.5827 (0.0587) &  1.0341 (0.0562)\\
		Group BPDN  & 0.0101 (0.0030)&  0.0057 (0.0043)& 0.0083 (0.0045) \\
		TV-SBL      & 0.1605 (0.0153)& 0.1341 (0.0038) & 0.7030 (0.0444) \\
		StructOMP   & 0.0104 (0.0022)     & 0.0056 (0.0020) &  0.0093 (0.0026)\\
		Mb-CS &  0.0005 (0.0014) & 0.0010 (0.0026) & 0.0006 (0.0006) \\
		RCS-SBL &  0.2980 (0.0166) &  0.2938 (0.0156)& 3.4150 (0.0882)\\
		DivSBL      & 0.1976 (0.0301) & 0.2365 (0.0260) & 0.6164 (0.0607) \\
		\rowcolor{customblue}
		SPP-SBL & 0.1147 (0.0104)&  0.1164 (0.0139)&  0.9420 (0.0332)\\
		\bottomrule
	\end{tabularx}
\end{table}

\subsection{Visualization of Boundary Estimation Results} \label{visualize beta}

The shrinkage of the Bayesian posterior variance indicates a contraction of the estimation boundary of the signal. In this section, we visually compare the behavior of three pattern-based methods (PC-SBL, A-EBSBL, RCS-SBL and proposed SPP-SBL) by presenting heatmaps of their learned posterior variances on the three types of data introduced above. To further illustrate the adaptive coupling capability of SPP-SBL, we visualize the learned $\vec{\boldsymbol{\beta}}$ values as well.

The first row of Fig. \ref{fig:beta} (a)–(c) presents heatmaps of the original signals for the heteroscedastic signal, chain signal, and multi-pattern sparse signal, respectively. All three exhibit clustered nonzero patterns, each with its own distinctive structure. The second to fifth rows display the posterior variance heatmaps learned by four algorithms on these three types of data. It can be clearly observed that PC-SBL, A-EBSBL and RCS-SBL tend to overestimate the posterior variance near the signal boundaries, failing to effectively truncate the support. In contrast, SPP-SBL yields significantly improved boundary estimates that more accurately reflect the true signal patterns. These observations are consistent with our discussion in Section~\ref{sec: discussion} and the support recovery rate (SRR) results reported in Tables \ref{tab:heteroscedastic}–\ref{tab:blockIBA-data}.

\begin{table*}[!t]
	\caption{Reconstructed error (RNMSE (std) \& Correlation (std)) of the test images.}
	\label{8images}
	\vskip 0.05in
	\setlength{\tabcolsep}{0.24cm}
	\centering
		\begin{tabular}{lcccccccc}
			\toprule
			& \textbf{Parrot} & \textbf{Cameraman} & \textbf{Lena} & \textbf{Boat} & \textbf{House} & \textbf{Barbara} & \textbf{Monarch} & \textbf{Foreman} \\ \cmidrule(lr){2-9}
			\textbf{Algorithm} & \multicolumn{8}{c}{\textbf{Square root of NMSE (std)}} \\
			\midrule
			{BSBL}           & 0.139 (0.004) & 0.156 (0.006) & 0.137 (0.004) & 0.179 (0.007) & 0.146 (0.007) & 0.142 (0.004) & 0.272 (0.009) & 0.125 (0.007) \\
			{PC-SBL}         & 0.133 (0.013) & 0.150 (0.012) & 0.134 (0.013) & 0.159 (0.014) & 0.137 (0.013) & 0.137 (0.013) & 0.208 (0.010) & 0.126 (0.014) \\
			{A-EBSBL}         & 0.156 (0.047) & 0.172 (0.041) & 0.143 (0.034) & 0.176 (0.025) & 0.158 (0.033) & 0.142 (0.026) & 0.252 (0.016) & 0.140 (0.040) \\
			{SBL}            & 0.225 (0.121) & 0.247 (0.141) & 0.223 (0.129) & 0.260 (0.114) & 0.238 (0.125) & 0.228 (0.119) & 0.458 (0.106) & 0.175 (0.099) \\
			{GLasso}    & 0.139 (0.017) & 0.153 (0.016) & 0.134 (0.017) & 0.159 (0.018) & 0.141 (0.018) & 0.135 (0.016) & 0.216 (0.020) & 0.124 (0.017) \\
			{GBPDN}     & 0.138 (0.017) & 0.153 (0.017) & 0.134 (0.017) & 0.159 (0.019) & 0.133 (0.019) & 0.135 (0.017) & 0.218 (0.022) & 0.123 (0.017) \\
			{TV-SBL}     & 0.214 (0.057) & 0.219 (0.044) & 0.206 (0.060) & 0.244 (0.058) & 0.212 (0.064) & 0.205 (0.051) & 0.382 (0.073) & 0.184  (0.055) \\
			{StrOMP}      & 0.161 (0.014) & 0.184 (0.013) & 0.159 (0.013) & 0.187 (0.014) & 0.162 (0.014) & 0.164 (0.013) & 0.248 (0.015) & 0.149 (0.016) \\
			Mb-CS & 0.160 (0.008) & 0.179 (0.009) & 0.148 (0.007) & 0.184 (0.011) & 0.156 (0.009) & 0.152 (0.008) & 0.274 (0.010) & 0.130 (0.011)\\
			{DivSBL}         & 0.117 (0.007) & 0.142 (0.006) & 0.114 (0.005) & 0.150 (0.008) & 0.120 (0.006) & 0.120 (0.005) & 0.203 (0.008) & 0.101 (0.007) \\
			
			\rowcolor{customblue}{SPP-SBL}         & \textbf{0.105} (0.008) & \textbf{0.124} (0.008) & \textbf{0.104} (0.007) & \textbf{0.133} (0.010) & \textbf{0.108} (0.008) & \textbf{0.109} (0.008) & \textbf{0.187} (0.008) & \textbf{0.094} (0.009) \\
			\cmidrule(lr){2-9}
			& \multicolumn{8}{c}{\textbf{Correlation (std)}} \\
			\midrule
			{BSBL}           & 0.940 (0.004) & 0.940 (0.005) & 0.923 (0.004) & 0.863 (0.012) & 0.887 (0.012) & 0.921 (0.004) & 0.768 (0.016) & 0.927 (0.009) \\
			{PC-SBL}         & 0.948 (0.010) & 0.948 (0.008) & 0.931 (0.012) & 0.902 (0.017) & 0.909 (0.018) & 0.931 (0.012) & 0.881 (0.011) & 0.930 (0.016) \\
			{A-EBSBL}         & 0.923 (0.055) & 0.929 (0.038) & 0.916 (0.044) & 0.871 (0.034) & 0.872 (0.043) & 0.921 (0.033) & 0.806 (0.025) & 0.910 (0.048) \\
			{SBL}            & 0.865 (0.111) & 0.870 (0.122) & 0.835 (0.138) & 0.792 (0.120) & 0.787 (0.143) & 0.837 (0.125) & 0.628 (0.092) & 0.874 (0.111) \\
			{GLasso}    & 0.946 (0.011) & 0.949 (0.008) & 0.935 (0.013) & 0.906 (0.017) & 0.909 (0.019) & 0.937 (0.012) & 0.874 (0.018) & 0.937 (0.015) \\
			{GBPDN}     & 0.947 (0.011) & 0.949 (0.008) & 0.936 (0.013) & 0.906 (0.017) & 0.910 (0.019) & 0.938 (0.012) & 0.871 (0.019) & 0.938 (0.015) \\
			{TV-SBL}     & 0.878 (0.055) & 0.898 (0.036) & 0.861 (0.062) & 0.814 (0.057) & 0.822 (0.070) & 0.870 (0.047) & 0.692 (0.068) & 0.870 (0.059) \\
			{StrOMP}      & 0.927 (0.012) & 0.924 (0.010) & 0.908 (0.014) & 0.873 (0.017) & 0.881 (0.019) & 0.906 (0.013) & 0.845 (0.015) & 0.907 (0.019) \\
			Mb-CS      & 0.923 (0.008) & 0.925 (0.008) & 0.913 (0.008) & 0.863 (0.015) & 0.878 (0.013) & 0.913 (0.008) & 0.780 (0.014) & 0.923 (0.013) \\
			{DivSBL}         & 0.959 (0.005) & 0.951 (0.004) & 0.948 (0.005) & 0.909 (0.010) & 0.927 (0.008) & 0.945 (0.005) & 0.882 (0.009) & 0.953 (0.007) \\
			
			\rowcolor{customblue}{SPP-SBL}         & \textbf{0.967} (0.005) & \textbf{0.964} (0.005) & \textbf{0.957} (0.006) & \textbf{0.929} (0.010) & \textbf{0.941} (0.009) & \textbf{0.955} (0.006) & \textbf{0.900} (0.008) & \textbf{0.959} (0.008) \\
			\bottomrule
		\end{tabular}
\end{table*}

\begin{figure*}[t]
	
	\hspace{0.02mm}
	\begin{subfigure}[b]{0.077\textwidth}
		\includegraphics[width=\textwidth]{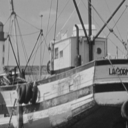}
	\end{subfigure}
	\begin{subfigure}[b]{0.077\textwidth}
		\includegraphics[width=\textwidth]{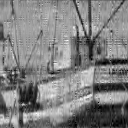}
	\end{subfigure}
	\begin{subfigure}[b]{0.077\textwidth}
		\includegraphics[width=\textwidth]{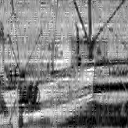}
	\end{subfigure}
	\begin{subfigure}[b]{0.077\textwidth}
		\includegraphics[width=\textwidth]{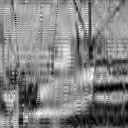}
	\end{subfigure}
	\begin{subfigure}[b]{0.077\textwidth}
		\includegraphics[width=\textwidth]{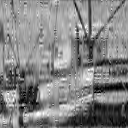}
	\end{subfigure}
	\begin{subfigure}[b]{0.077\textwidth}
		\includegraphics[width=\textwidth]{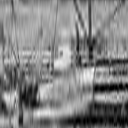}
	\end{subfigure}
	\begin{subfigure}[b]{0.077\textwidth}
		\includegraphics[width=\textwidth]{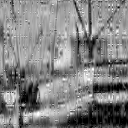}
	\end{subfigure}
	\begin{subfigure}[b]{0.077\textwidth}
		\includegraphics[width=\textwidth]{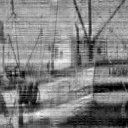}
	\end{subfigure}
	\begin{subfigure}[b]{0.077\textwidth}
		\includegraphics[width=\textwidth]{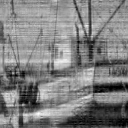}
	\end{subfigure}
	\begin{subfigure}[b]{0.077\textwidth}
		\includegraphics[width=\textwidth]{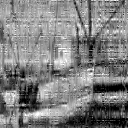}
	\end{subfigure}
	\begin{subfigure}[b]{0.077\textwidth}
		\includegraphics[width=\textwidth]{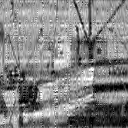}
	\end{subfigure}
	\begin{subfigure}[b]{0.077\textwidth}
		\includegraphics[width=\textwidth]{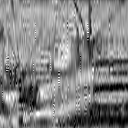}
	\end{subfigure}
	
	\vspace{1mm}
	
	\hspace{0.02mm}
	\begin{subfigure}[b]{0.077\textwidth}
		\includegraphics[width=\textwidth]{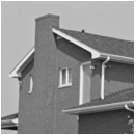}
	\end{subfigure}
	\begin{subfigure}[b]{0.077\textwidth}
		\includegraphics[width=\textwidth]{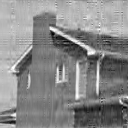}
	\end{subfigure}
	\begin{subfigure}[b]{0.077\textwidth}
		\includegraphics[width=\textwidth]{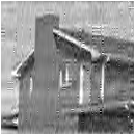}
	\end{subfigure}
	\begin{subfigure}[b]{0.077\textwidth}
		\includegraphics[width=\textwidth]{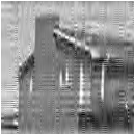}
	\end{subfigure}
	\begin{subfigure}[b]{0.077\textwidth}
		\includegraphics[width=\textwidth]{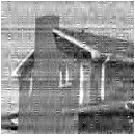}
	\end{subfigure}
	\begin{subfigure}[b]{0.077\textwidth}
		\includegraphics[width=\textwidth]{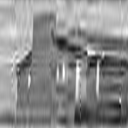}
	\end{subfigure}
	\begin{subfigure}[b]{0.077\textwidth}
		\includegraphics[width=\textwidth]{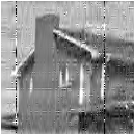}
	\end{subfigure}
	\begin{subfigure}[b]{0.077\textwidth}
		\includegraphics[width=\textwidth]{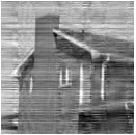}
	\end{subfigure}
	\begin{subfigure}[b]{0.077\textwidth}
		\includegraphics[width=\textwidth]{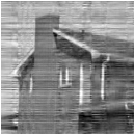}
	\end{subfigure}
	\begin{subfigure}[b]{0.077\textwidth}
		\includegraphics[width=\textwidth]{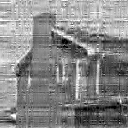}
	\end{subfigure}
	\begin{subfigure}[b]{0.077\textwidth}
		\includegraphics[width=\textwidth]{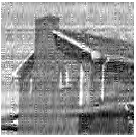}
	\end{subfigure}
	\begin{subfigure}[b]{0.077\textwidth}
		\includegraphics[width=\textwidth]{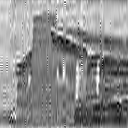}
	\end{subfigure}
	
	\vspace{1mm}
	
	\hspace{0.02mm}
	\begin{subfigure}[b]{0.077\textwidth}
		\includegraphics[width=\textwidth]{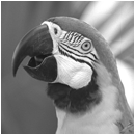}
		\caption{\scriptsize Original}
	\end{subfigure}
	\begin{subfigure}[b]{0.077\textwidth}
		\includegraphics[width=\textwidth]{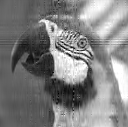}
		\caption{SPP-SBL}
	\end{subfigure}
	\begin{subfigure}[b]{0.077\textwidth}
		\includegraphics[width=\textwidth]{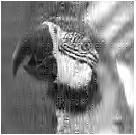}
		\caption{DivSBL}
	\end{subfigure}
	\begin{subfigure}[b]{0.077\textwidth}
		\includegraphics[width=\textwidth]{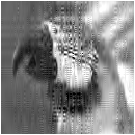}
		\caption{BSBL}
	\end{subfigure}
	\begin{subfigure}[b]{0.077\textwidth}
		\includegraphics[width=\textwidth]{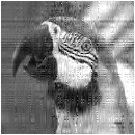}
		\caption{PC-SBL}
	\end{subfigure}
	\begin{subfigure}[b]{0.077\textwidth}
		\includegraphics[width=\textwidth]{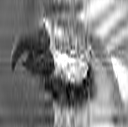}
		\caption{A-EBSBL}
	\end{subfigure}
	\begin{subfigure}[b]{0.077\textwidth}
		\includegraphics[width=\textwidth]{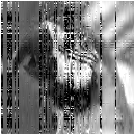}
		\caption{SBL}
	\end{subfigure}
	\begin{subfigure}[b]{0.077\textwidth}
		\includegraphics[width=\textwidth]{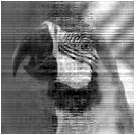}
		\caption{GLasso}
	\end{subfigure}
	\begin{subfigure}[b]{0.077\textwidth}
		\includegraphics[width=\textwidth]{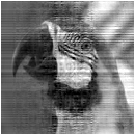}
		\caption{GBPDN}
	\end{subfigure}
	\begin{subfigure}[b]{0.077\textwidth}
		\includegraphics[width=\textwidth]{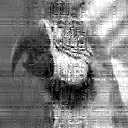}
		\caption{TV-SBL}
	\end{subfigure}
	\begin{subfigure}[b]{0.077\textwidth}
		\includegraphics[width=\textwidth]{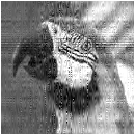}
		\caption{StrOMP}
	\end{subfigure}
	\begin{subfigure}[b]{0.077\textwidth}
		\includegraphics[width=\textwidth]{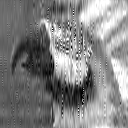}
		\caption{Mb-CS}
	\end{subfigure}
	
	\caption{Reconstruction results for Boat, House and Parrot images.}
	\label{fig_results}
\end{figure*}

Furthermore, by examining the heatmaps of the posterior of $\vec{\boldsymbol{\beta}}$ (as shown in the sixth row of Fig. \ref{fig:beta} (a)–(c)), we observe that most entries converge to a common mean (e.g., around $9$ in Fig. \ref{fig:beta}), while deviations from this mean occur only in regions corresponding to true signal variation. As a result, the heatmaps of $\vec{\boldsymbol{\beta}}$ closely mirror the structural patterns of the ground-truth signals. This further supports our claim that the key to improved recovery performance lies in the relative values of $\vec{\boldsymbol{\beta}}$ learned by the model, which adaptively capture the underlying signal structure.

\subsection{Audio Dataset}\label{audioset}

Audio signals exhibit block sparse structures under the discrete cosine transform (DCT) basis, as illustrated in Appendix~\ref{App:block}, making them particularly suitable for evaluating block-sparse recovery algorithms. In this subsection, we conduct experiments on real-world audio samples randomly selected from the AudioSet dataset \cite{gemmeke2017audio}, focusing on the phase transition behavior of SPP-SBL in comparison with other block-sparse recovery methods.

The experimental setup follows the configuration described in \cite{NEURIPS2024_ead542f1}. Specifically, each audio signal contains approximately 90 non-zero coefficients in the DCT domain (i.e., $K = 90$), accounting for about $20\%$ of the total dimensionality ($N = 480$). Accordingly, we begin with a sampling rate of $20\%$, where the ratio $M/K$ is close to 1 and increases with higher sampling rates $M/N$ (from 0.2 to 0.55). Meanwhile, the signal-to-noise ratio (SNR) is gradually varied from 10~dB to 40~dB.

The square root of NMSE for each algorithm is shown in Fig.~\ref{phase}, where darker colors in the heatmap indicate lower reconstruction errors. SPP-SBL demonstrates a clear advantage over the other methods, particularly under extreme measurement ratios and low SNR conditions. Its phase transition curve spans a noticeably larger region of successful recovery, reflecting stronger robustness and reconstruction capability. 

We further evaluated SPP-SBL on high-dimensional data using this dataset. The results demonstrate that SPP-SBL scales well to high dimension while maintaining robust and stable reconstruction performance. Detailed results are provided in Appendix~\ref{App:high_dimen}.

\subsection{Image Reconstruction}\label{image}

Images tend to exhibit block-sparse structures in the discrete wavelet transform (DWT) domain, as shown in Appendix~\ref{App:block}.
In this experiment, we evaluate block-sparse recovery algorithms using a standard set of grayscale images\footnote{The above eight grayscale images are widely used benchmark set for image reconstruction, available at http://dsp.rice.edu/software/DAMP-toolbox and http://see.xidian.edu.cn/faculty/wsdong/NLR\_Exps.htm.}. A sampling rate of 0.5 is applied, and the reconstruction errors and correlations are summarized in Table \ref{8images}. Among all compared methods, SPP-SBL achieves the most accurate reconstructions in a statistically significant sense. As shown in Fig. \ref{fig_results}, which presents the reconstructed Boat, House and Parrot images as representative examples, SPP-SBL effectively suppresses artifact streaks and produces noticeably smoother and more visually faithful results.

\section{Conclusion} \label{conclusion}

This paper introduced a unified variance transformation framework that incorporated a space-power prior based on undirected graphs to adaptively capture unknown block structures while directly estimating their coupling parameters. We proved that each coupling parameter $\beta_i$ admitted a unique positive real solution within this framework. Building on this theoretical foundation, we developed SPP-SBL, an EM-based algorithm enhanced with high-order equation root-finding techniques, which ultimately resolved the longstanding challenge of coupling parameter estimation in pattern-based block sparse recovery methods. Notably, our analysis revealed that the relative magnitudes of the learned coupling parameters $\vec{\boldsymbol{\beta}}$, rather than tuning a single shared value, played a key role in enhancing recovery performance.
Extensive simulations involving complex structural patterns, such as chain-structured and multi-pattern sparse signals, along with real-world experiments on multi-modal data (images and audio), demonstrated that SPP-SBL significantly outperformed existing methods and effectively captured the underlying structures of the signals.

Future work may explore richer coupling matrix structures $\boldsymbol{T}$ within the variance‑transformation framework to accommodate more diverse block sparse patterns, and extend the space‑power prior to two‑ or even higher‑dimensional settings for applications in high‑dimensional imaging, video, and spatiotemporal signal processing.




%

\appendices
\section{Visualization of Structured Sparsity in Transform Domains}\label{App:block}

\begin{figure}[h]
	\begin{center}
		\includegraphics[width=3.3in]{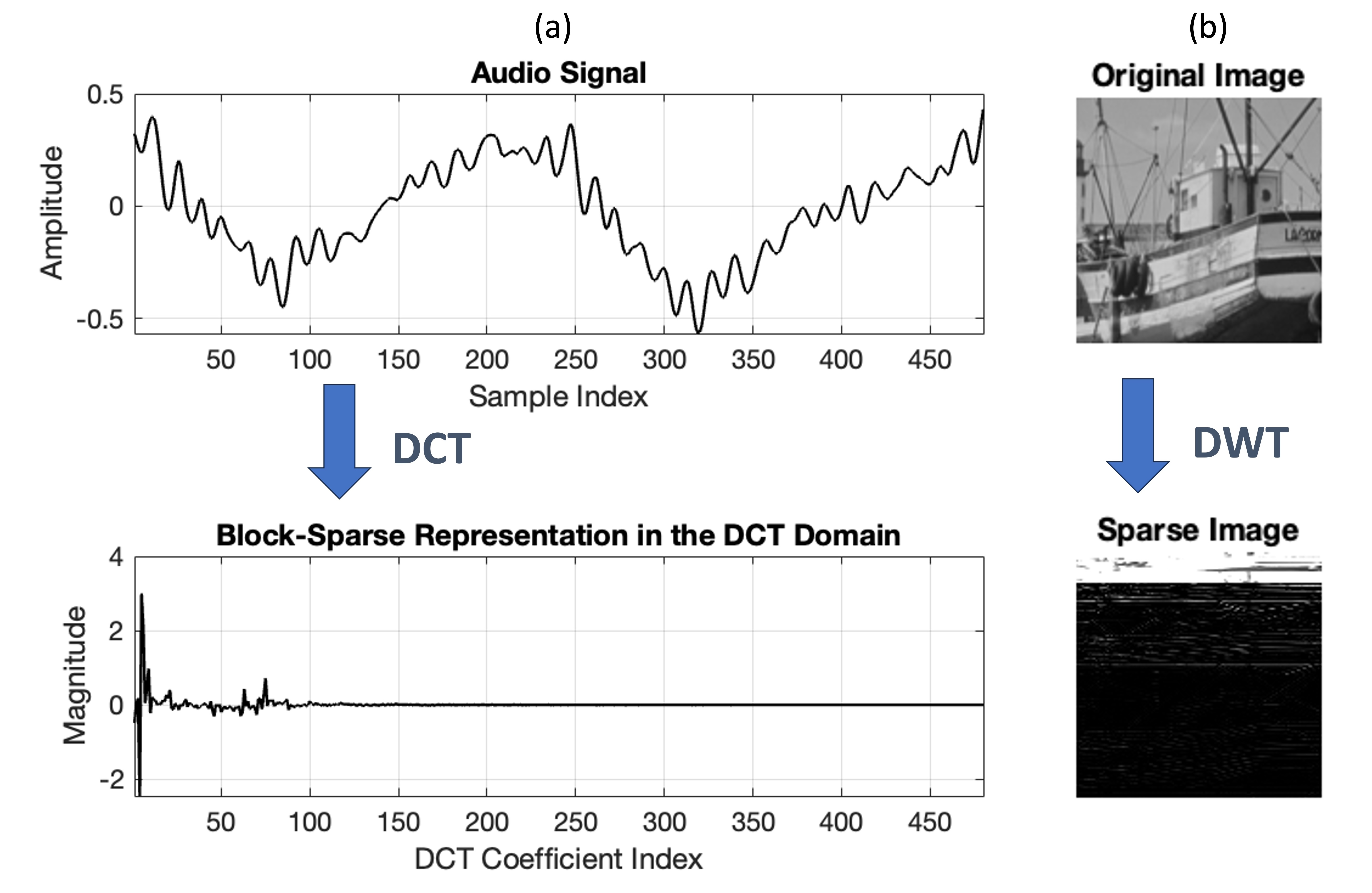}
	\end{center}
	\caption{Block sparse representations in the corresponding transform domains.}
	\label{block sparse}
\end{figure}

Many real-world signals admit block-sparse representations after appropriate transforms. 
Specifically, natural images in the discrete wavelet transform (DWT) domain 
and audio signals in the discrete cosine transform (DCT) domain 
naturally exhibit significant coefficients that form contiguous clusters, 
rather than isolated nonzero entries. 
As shown in Fig.~\ref{block sparse}, the audio and image signals used in 
Section~\ref{audioset} and Section~\ref{image} display clear block-sparse structures in their corresponding transform domains.

\section{Computational complexity of SPP-SBL}\label{APP: complexity}

The computational complexity of each component of the proposed SPP-SBL algorithm is summarized in Table~\ref{tab:complexity}. The overall per-iteration computational complexity is
$\mathcal{O}(M^2N)$,
which is identical to that of PC-SBL, A-EBSBL and original SBL \cite{ji2008bayesian}. 

\begin{table}[h]
	\centering
	\caption{Per-iteration computational complexity of SPP-SBL}
	\label{tab:complexity}
	\begin{tabular}{l r}
		\hline
		\textbf{Step} & \textbf{Complexity} \\
		\hline
		Posterior mean and covariance update ($\boldsymbol{\mu}, \boldsymbol{\Sigma}$) 
		& $\mathcal{O}(M^2N)$ \\
		Hyperparameter update ($\boldsymbol{\alpha}$) 
		& $\mathcal{O}(N)$ \\
		Coupling parameter update ($\vec{\boldsymbol{\beta}} $) 
		& $\mathcal{O}(N)$ \\
		Noise precision update ($\gamma$) 
		& $\mathcal{O}(MN)$ \\
		\hline
		\textbf{Total per iteration} 
		& $\boldsymbol{\mathcal{O}(M^2N)}$ \\
		\hline
	\end{tabular}
\end{table}

Although SPP-SBL introduces an additional update step for the coupling parameters $\vec{\boldsymbol{\beta}}$, this step only incurs linear complexity in $N$ and is therefore negligible compared to the dominant cost $\mathcal{O}(M^2N)$, i.e., posterior mean and covariance update, which is also involved by almost all SBL-type algorithms.
Moreover, jointly learning $\boldsymbol{\alpha}$ and $\vec{\boldsymbol{\beta}}$ provides a more expressive prior model and empirically leads to faster convergence, as illustrated in Fig.~\ref{time}.

\begin{figure}[ht]
	\begin{center}
		\includegraphics[width=2.3in]{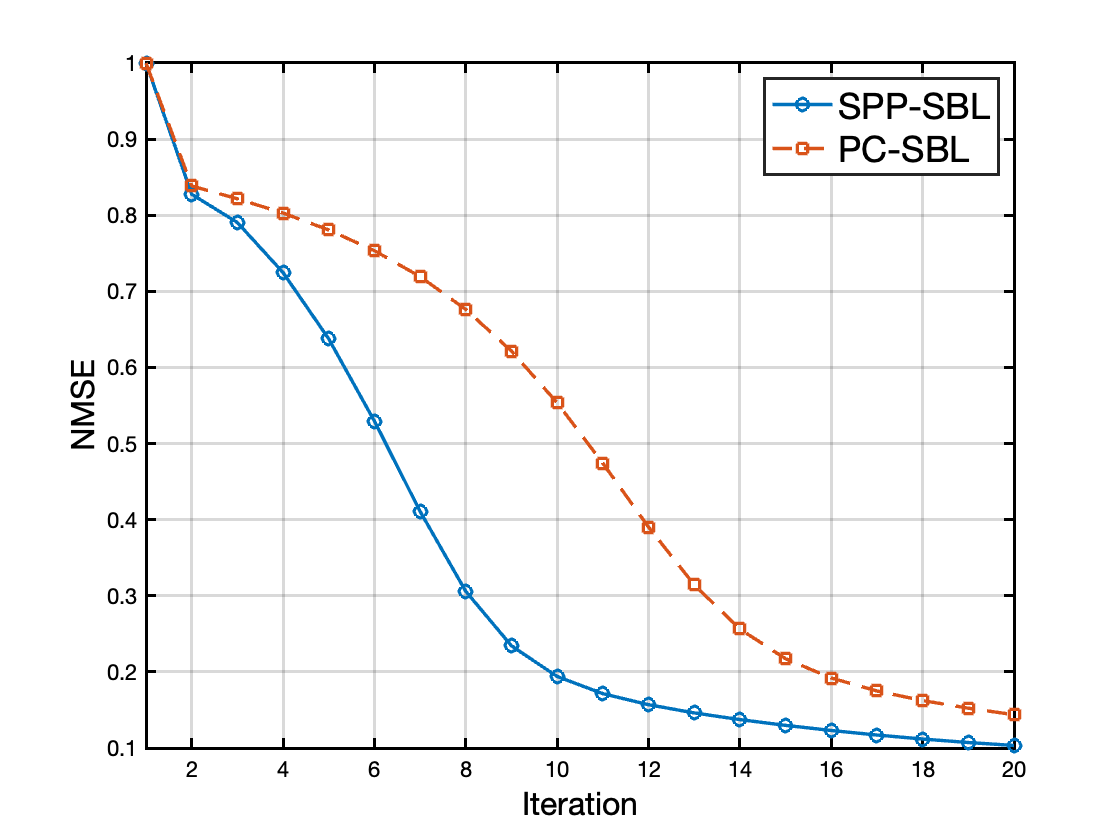}
	\end{center}
	\caption{Convergence result of SPP-SBL and PC-SBL on audio signal}
	\label{time}
\end{figure}

It's observed  that SPP-SBL requires nearly half the number of iterations compared with PC-SBL to reach the same level of reconstruction accuracy. Therefore, SPP-SBL not only has the same per-iteration computational complexity as PC-SBL in theory, but also exhibits clear advantages in terms of practical runtime efficiency and estimation accuracy. A detailed runtime comparison of all considered algorithms on different types of data is further provided in Table~\ref{tab:runtime_all}.

\section{The sensitivity to initialization} \label{App:init}

To evaluate the sensitivity of SPP-SBL to hyperparameter initialization, we performed experiments on four groups of parameters on heteroscedastic block-sparse data: 1. Coupling parameter $\vec{\boldsymbol{\beta}} = \beta \cdot \text{ones}(N-1,1)$, 2. Variance-layer hyperparameters ${\boldsymbol{\alpha}} = \alpha \cdot \text{ones}(N,1)$, 3. Noise precision parameter $\gamma$, 4. Prior parameters $c/d$. 

Fig.~\ref{fig:four_images} shows the reconstruction performance under different initializations for each group. It can be observed that the initialization of the noise precision $\gamma$ may slightly affect the convergence speed, but the algorithm consistently converges to the same solution. The other three groups of parameters ($\vec{\boldsymbol{\beta}}, {\boldsymbol{\alpha}}, c/d$) show minimal sensitivity, with negligible impact on the final performance.

\begin{figure}[h]
	\centering
	\begin{subfigure}[b]{0.48\linewidth}  
		\includegraphics[width=\linewidth]{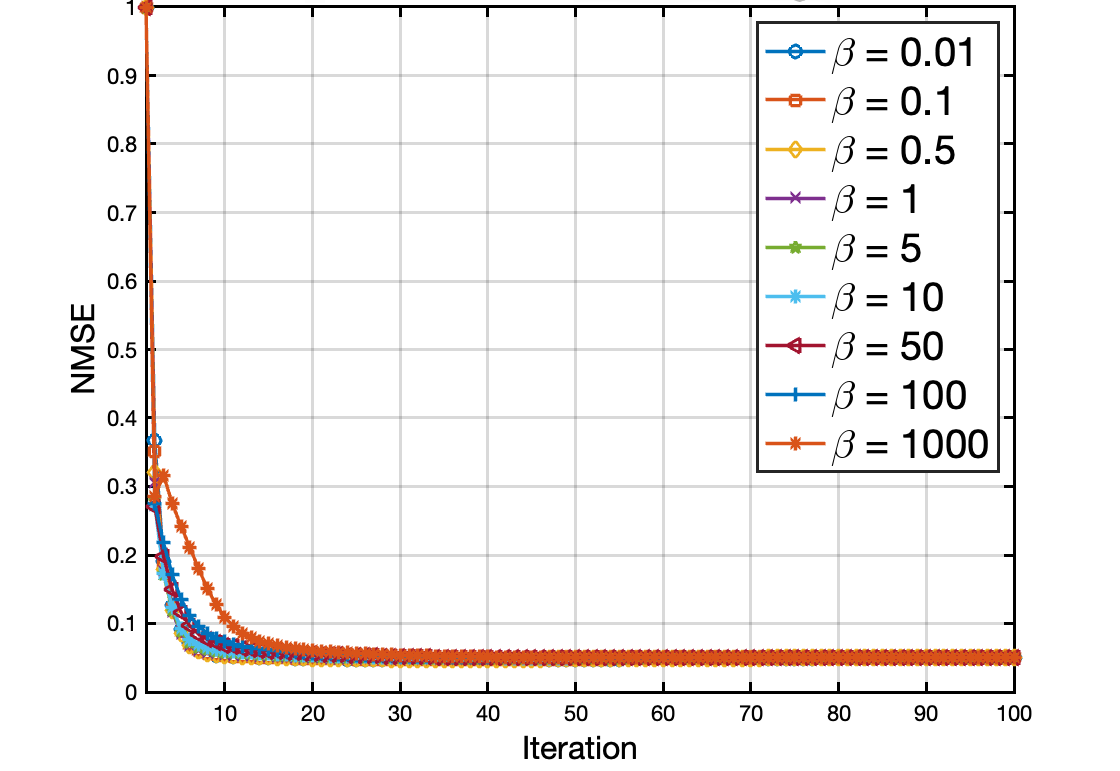}  
		\caption{Coupling parameter $\vec{\boldsymbol{\beta}}$}
		\label{fig:init_beta}
	\end{subfigure}
	\hfill  
	\begin{subfigure}[b]{0.48\linewidth}
		\includegraphics[width=\linewidth]{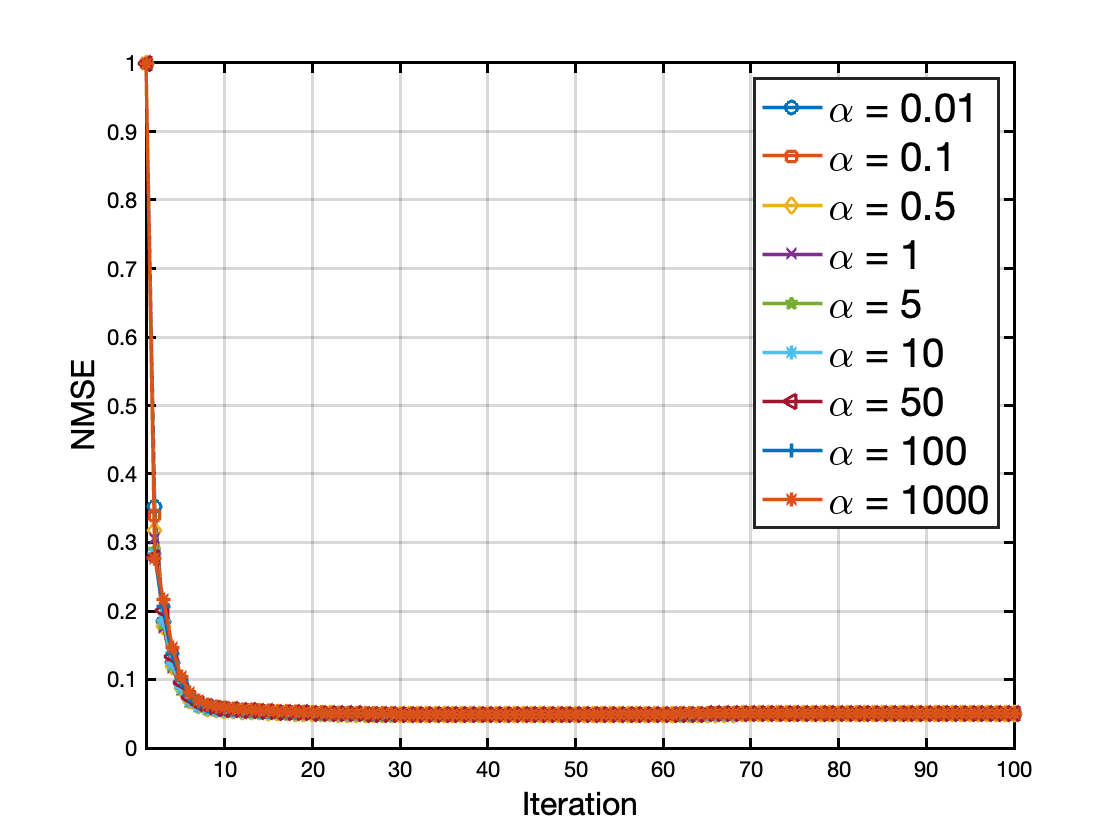}
		\caption{Variance-layer hyperparameters ${\boldsymbol{\alpha}}$}
		\label{fig:init_alpha}
	\end{subfigure}
	
	\vspace{0.5em}  
	
	\begin{subfigure}[b]{0.48\linewidth}
		\includegraphics[width=\linewidth]{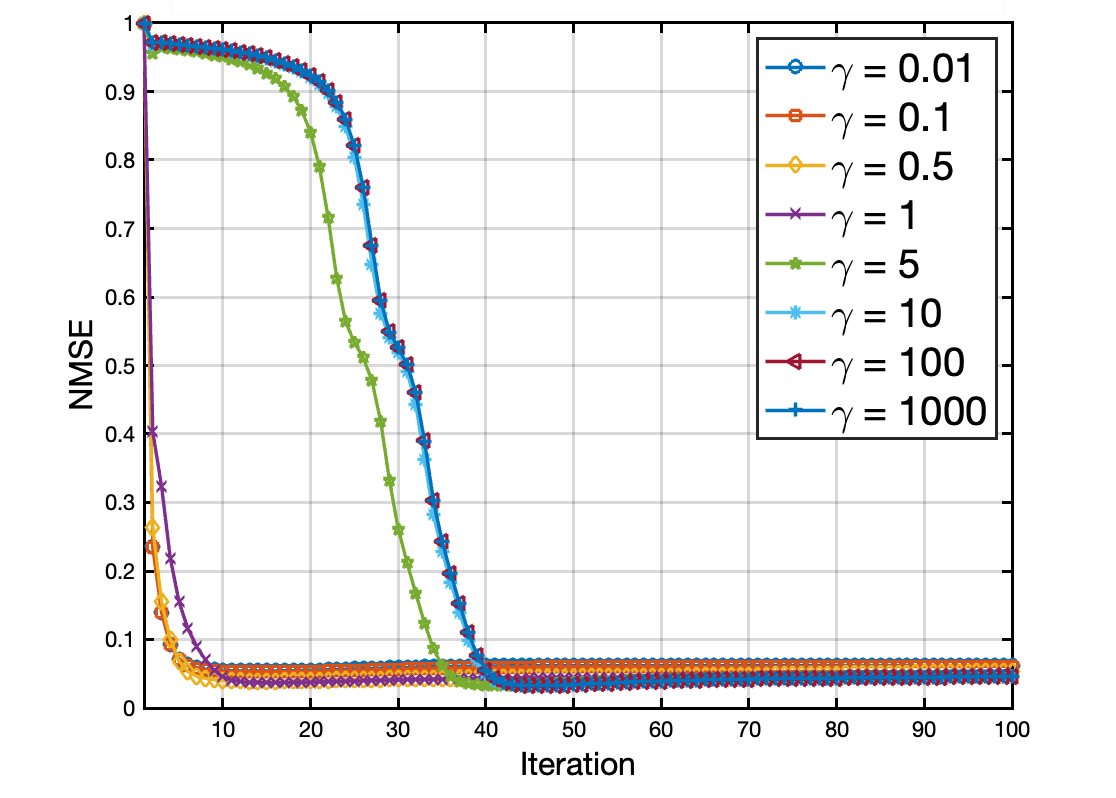}
		\caption{Noise precision $\gamma$}
		\label{fig:init_noise}
	\end{subfigure}
	\hfill
	\begin{subfigure}[b]{0.48\linewidth}
		\includegraphics[width=\linewidth]{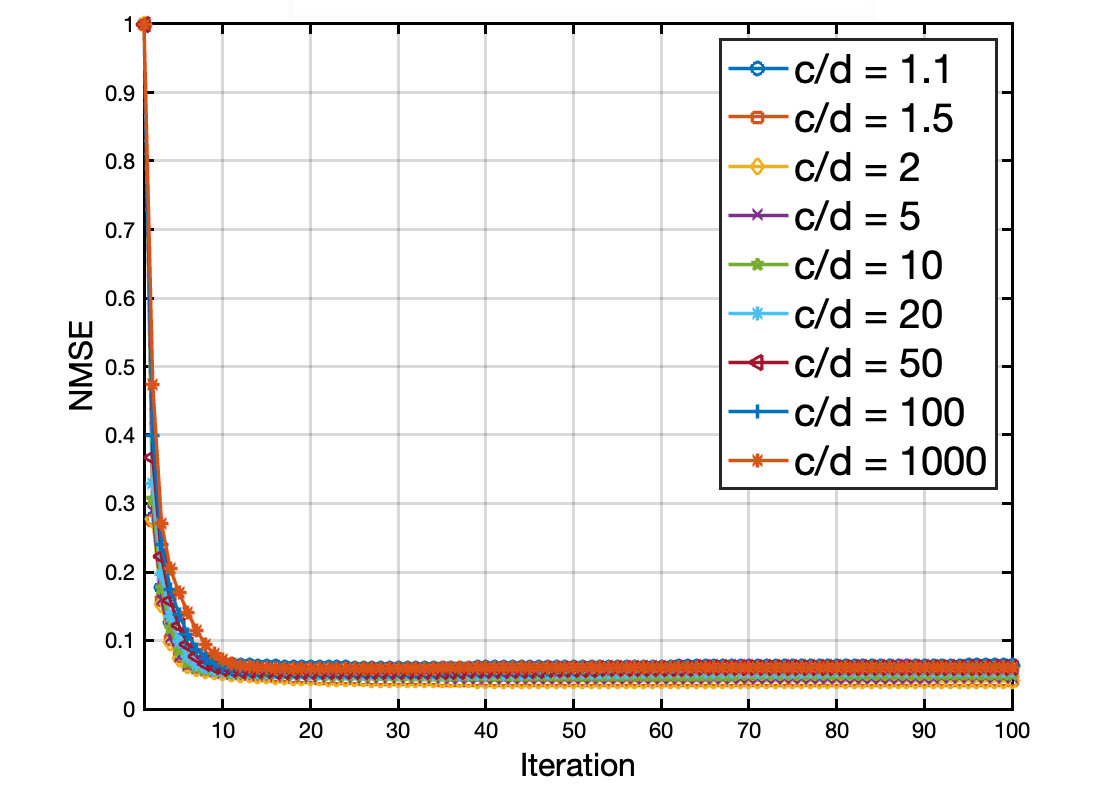}
		\caption{Prior parameters $c/d$}
		\label{fig:init_c_d}
	\end{subfigure}
	
	\caption{Sensitivity analysis of SPP-SBL to hyperparameter initialization.}
	\label{fig:four_images}
\end{figure}

\section{Scalability to high-dimensional signals}\label{App:high_dimen}

To further evaluate the scalability of the proposed SPP-SBL algorithm, we conducted additional experiments on audio data by systematically increasing the signal dimension. Specifically, this was achieved by extracting longer audio segments with lengths
$N = 960, 2400, 4800,$ and $9600$, corresponding to $2\times, 5\times, 10\times,$ and $20\times$ the original dimension ($N = 480$) used in the main experiment in Section~\ref{audioset}. The signal-to-noise ratio (SNR) and compression ratio were fixed at $15$~dB and $25\%$, respectively, throughout these tests.

Table~\ref{tab:dimension} summarizes the quantitative NMSE results across different signal dimensions. As observed, SPP-SBL consistently achieves the lowest reconstruction error and maintains a clear performance advantage over competing methods as the signal dimension increases. These results indicate that the proposed SPP-SBL algorithm scales well to high-dimensional signals while preserving robust and stable reconstruction performance. 
\begin{table}[t]
	\centering
	\caption{Performance comparison (NMSE) across different Audioset signal dimensions $N$. 
		Our method is highlighted in \colorbox{customblue}{blue} and the best results are shown in \textbf{bold}.}
	\label{tab:dimension}
	\small
	\renewcommand{\arraystretch}{1.0}
	\setlength{\tabcolsep}{3pt}
	\begin{tabular}{lcccc}
		\toprule
		\textbf{Algorithm} 
		& $\mathbf{N=960}$ 
		& $\mathbf{N=2400}$ 
		& $\mathbf{N=4800}$ 
		& $\mathbf{N=9600}$ \\
		\midrule
		BSBL         & 0.3132 & 0.2289& 0.1824 & 0.2184 \\
		PC-SBL       & 0.2852 &0.2326 & 0.2287 & 0.2436 \\
		Block-IBA    & 0.2792 & 0.2804& 0.2662& 0.3093\\
		SBL          & 0.3429 & 0.2753& 0.2694 & 0.3222 \\
		Group Lasso  & 0.3924 & 0.2771 & 0.2178& 0.2408 \\
		Group BPDN   & 0.4001 & 0.2817 & 0.2252 & 0.2484 \\
		StructOMP    & 0.2693 & 0.1800 & 0.1965& 0.3796 \\
		Mb-CS       & 0.2630 & 0.2822 & 0.2609 & 0.4760 \\
		RCS-SBL     & 0.2563 & 0.2058 & 0.1954 & 0.2140 \\
		DivSBL      & 0.1992 & 0.2051 & 0.1752 &0.2063 \\
		\rowcolor{customblue}
		SPP-SBL     & \textbf{0.1952} & \textbf{0.1535} & \textbf{0.1404}& \textbf{0.1755} \\
		\bottomrule
	\end{tabular}
\end{table}

\vfill

\end{document}